\documentclass{amsart}
\usepackage{amsxtra, microtype}
\usepackage{stmaryrd }
\usepackage{mathrsfs}
\usepackage{amssymb,amsmath,amsthm,stmaryrd,latexsym,wasysym}
\usepackage[all]{xy}
\usepackage{hyperref}
\usepackage{cancel}
\newcommand{\tilo}{\widetilde 0} 
\theoremstyle{plain}
 
\usepackage{xcolor}

\usepackage{xspace}

\newtheorem{theorem}{Theorem}[section]
\newtheorem{lemma}[theorem]{Lemma}
\newtheorem{proposition}[theorem]{Proposition}
\newtheorem{corollary}[theorem]{Corollary}
\newtheorem{definition}[theorem]{Definition} \theoremstyle{definition}
\newtheorem{example}[theorem]{Example}
\newtheorem{remark}[theorem]{Remark}

\newcommand{\duer}{\mathbin{\raisebox{3pt}{\varhexstar}\kern-3.70pt{\rule{0.15pt}{4pt}}}\,}

\newcommand{\R}{\mathbb{R}}

\newcommand{\mx}{\mathfrak{X}} \newcommand{\dr}{\mathbf{d}}

\newcommand{\ldr}[1]{{{\pounds}}_{#1}}

\DeclareMathOperator{\Hom}{Hom} 

\DeclareMathOperator{\Id}{Id}

\newcommand{\co}{\colon}
\newcommand{\Ga}{\Gamma}
\newcommand{\til}[1]{\widetilde{#1}}
\newcommand{\eps}{\varepsilon}

\newcommand{\nsp}[2]  {%
\langle\mspace{-6.8mu}%
\langle\mspace{-6.8mu}%
\langle\mspace{-6.8mu}%
\langle\mspace{-6.8mu}%
\langle\mspace{-6.8mu}%
\langle\mspace{-6.8mu}%
\langle{#1,\,}{#2}%
\rangle%
\mspace{-6.8mu}\rangle%
\mspace{-6.8mu}\rangle%
\mspace{-6.8mu}\rangle%
\mspace{-6.8mu}\rangle%
\mspace{-6.8mu}\rangle%
\mspace{-6.8mu}\rangle}

\newcommand{\VBgpd}{$\mathcal{V}\!\mathcal{B}$--group\-oid\xspace}

\newcommand{\LAgpds}{$\mathcal{L}\!\mathcal{A}$--group\-oids\xspace}
\newcommand{\LAvb}{$\mathcal{L}\!\mathcal{A}$--vector bundle\xspace}
\newcommand{\LAvbs}{$\mathcal{L}\!\mathcal{A}$--vector bundles\xspace}

\newcommand{\relphantom}[1]{\mathrel{\phantom{#1}}}

\begin{document}
%%%%%%%%%%%%%%%%%%%%%%%%%%%%%%%%%%%%%%%%%%%%%%%%%%%%%%%%%%%%%%%%%%%%%%%%%%%
%%%%%%%%%%%%%%%%%%%%%%    Title    %%%%%%%%%%%%%%%%%%%%%%%%%%%%%%%%%%%%%%%%
\title[Double Lie algebroids and representations up to homotopy]{Double Lie 
algebroids\\ and representations up to homotopy}

%%% author one information
\author{A. Gracia-Saz}
\address{Department of Mathematics, University of Toronto.}
\email{alfonso@math.toronto.edu}
\author{M. Jotz Lean}
\address{School of Mathematics and Statistics, The University of Sheffield.} 
\email{M.Jotz-Lean@sheffield.ac.uk}
\author{K.~C.~H. Mackenzie}
\address{School of Mathematics and Statistics, The University of Sheffield.} 
\email{K.Mackenzie@sheffield.ac.uk}
\author{R. A. Mehta}
\address{Department of Mathematics \& Statistics\\
44 College Lane\\
Northampton, MA 01063}
\email{rmehta@smith.edu}
% \email{rajan.a.mehta@gmail.com}
\thanks{Research partially supported  by a \emph{fellowship for prospective researchers (PBELP2\_137534) of the
Swiss NSF} for a postdoctoral stay at UC Berkeley, and by a Dorothea-Schl\"ozer fellowship of the University of G\"ottingen.}


\begin{abstract}
  We show that double Lie algebroids, together with a chosen linear
  splitting, are equivalent to pairs of $2$-term representations up to
  homotopy satisfying compatibility conditions which extend the notion
  of matched pair of Lie algebroids. We discuss in detail the tangent
  of a Lie algebroid.
\end{abstract}
\maketitle

\tableofcontents

% \newpage

\section{Introduction}
\label{sect:intro}

Double Lie algebroids first arose as the infinitesimal form of double Lie 
groupoids \cite{Mackenzie92,Mackenzie00}. In the same way as the Lie theory 
of Lie groupoids and Lie algebroids expresses many of the basic 
infinitesimalization and integration results of differential geometry, 
the process of taking the double Lie algebroid of a double Lie groupoid 
captures two-stage differentiation processes, such as the iterated tangent 
bundle of a smooth manifold, and the relations between a Poisson Lie group, 
its Lie bialgebra and its symplectic double groupoid. 

The transition from a double Lie groupoid to its double Lie algebroid is
straightforward. To define an abstract concept of double Lie algebroid, 
however, is much more difficult, since there is no meaningful way in which 
a Lie algebroid bracket can be said to be a morphism with respect to another 
Lie algebroid structure. The solution ultimately found was to extend the 
duality between  Lie algebroids and Lie-Poisson structures to the double 
context, using the duality properties of double vector bundles \cite{Mackenzie11}. 
This definition was immediately given a simple and elegant reformulation 
in terms of super geometry and coordinates by Th.~Voronov \cite{Voronov12}. 
In terms of super geometry, a Lie algebroid structure on a vector bundle 
corresponds to a homological vector field $Q$ of weight $1$ on the 
parity-reversed bundle. A double vector bundle $D$ with Lie algebroid 
structures on both bundle structures on $D$ therefore involves two 
homological vector fields, of suitable weights, and the main 
compatibility condition of \cite{Voronov12} is that they commute. 

In the present paper we give a third formulation of double Lie algebroids, 
in terms of representations up to homotopy as defined in \cite{GrMe10a} and 
\cite{ArCr12}; this differs from the concept of \cite{EvLuWe99}.
In fact the representations up to homotopy which are relevant are concentrated 
in degrees $0$ and $1$ and we refer to them as \emph{$2$-representations} for 
brevity. Consider first a double vector bundle $D$ with Lie algebroid structures 
on two parallel sides, which are compatible with the vector bundle structures 
on the other sides; these are variously called \LAgpds or VB-groupoids, and 
are the `preliminary case' of double Lie algebroids in \cite{Mackenzie11}. 
Further suppose that $D$ is `decomposed'; that is, as a manifold it is the
fibre product $A\times _M B\times_M C$ of three vector bundles $A$, $B$, $C$ 
on the same base $M$, and the vector bundle structures on $D$ are the
pullbacks of $B\oplus C$ to $A$ and of $A\oplus C$ to $B$. Then \cite{GrMe10a} 
showed that VB-groupoid structures on $D$ are in bijective correspondence with 
$2$-representations defined in terms of $A$, $B$ and $C$. 

It is always possible to `decompose' a double vector bundle; 
that is, any double vector bundle is isomorphic to a decomposed double vector
bundle. Decompositions may be regarded as trivializations of $D$ at the double 
level; in this paper we do not need to trivialize $A$, $B$ and $C$. 
For a formulation in coordinate terms, see \cite{Voronov12}.

Now consider an arbitrary double Lie algebroid $D$. The Lie algebroid structures
on $D$ may be considered as a pair of VB-structures and accordingly a decomposition 
of $D$ expresses the double Lie algebroid structure as a pair of $2$-representations. 
Our main result (Theorem \ref{main}) determines the compatibility conditions between
these and, conversely, proves that a suitable pair of $2$-representations defines a 
double Lie algebroid structure on $D$. 

Our formulation is in some respects midway between the original formulation
and that of \cite{Voronov12}. Our treatment resembles that of Voronov inasmuch 
as the three intrinsic conditions of \cite{Mackenzie11} are replaced by a greater 
number of conditions which are dependent on auxiliary data, but are easier 
to work with. On the other hand, our methods are entirely `classical' rather 
than supergeometric, and rely on a global decomposition rather than local
coordinates.

Our formulation may also be regarded as a considerable generalization
of the description of a vacant double Lie algebroid in terms of a matched 
pair of representations \cite[\S6]{Mackenzie11} --- that is, of 
representations of Lie algebroids in the strict sense, 
without curvature. For this reason we regard the conditions (M1) to (M9) 
in Definition \ref{matched} as defining a matched pair of 2-representations. 

In turn, \cite{Jotz14b} will show that the bicrossproduct of a
  matched pair of 2-representations is a split Lie 2-algebroid, in the same way that the
  bicrossproduct of a matched pair of representations of Lie
  algebroids is a Lie
  algebroid \cite{Mokri97,Lu97}.
In a different direction, \cite{JoMa14b} will apply our main result
to show that double Lie algebroids 
which are transitive in a sense appropriate to the double structure are 
determined by a simple diagram of morphisms of ordinary Lie algebroids. 

We now describe the contents of the paper. 

In \S\ref{sect:b+d} we recall the basic notions needed throughout the paper. 
We begin with double vector bundles, the special classes of sections 
with which it is easiest to work, and the nonstandard pairing between their
duals. In \S\ref{subsect:VBa} we recall VB-algebroids and double Lie algebroids, 
and in \S\ref{subsect:ruths} we finally define $2$-representations.  

In \S\ref{sect:mte} we state our main result and the main work of the proof
is given in \S\ref{sect:proof}. 

We have included definitions of the key concepts required; in particular it
is not necessary to have a detailed knowledge of \cite{ArCr12}, \cite{GrMe10a}, 
\cite{Mackenzie11} or \cite{Voronov12}. 

\section{Background and definitions}
\label{sect:b+d}
\subsection{Double vector bundles, decompositions and dualization}
We briefly recall the definitions of double vector bundles, of their
\emph{linear} and \emph{core} sections, and of their \emph{linear
splittings} and \emph{lifts}. We refer to
\cite{Pradines77,Mackenzie05,GrMe10a} for more detailed treatments.

\begin{definition}
\label{df:dvb}
A \textbf{double vector bundle} is a commutative square
\begin{equation*}
\begin{xy}
\xymatrix{
D \ar[r]^{\pi_B}\ar[d]_{\pi_A}& B\ar[d]^{q_B}\\
A\ar[r]_{q_A} & M}
\end{xy}
\end{equation*}
satisfying the following four conditions:
\begin{enumerate}
\item all four sides are vector bundles;
\item $\pi_B$ is a vector bundle morphism over $q_A$;
\item $+_B: D\times_B D \rightarrow D$ is a vector bundle
morphism over $+: A\times_M A \rightarrow A$, where $+_B$ is
the addition map for the vector bundle $D\rightarrow B$, and 
\item the scalar multiplication $\R\times D\to D$ in the bundle
$D\to B$ is a vector bundle morphism over the scalar multiplication
$\R\times A\to A$. 
\end{enumerate}
\end{definition}

The corresponding statements for the operations in the bundle $D\to A$ follow. 

Given a double vector bundle $(D; A, B; M)$, the vector bundles $A$ and $B$ are 
called the \textbf{side bundles}. The \textbf{core} $C$ of a double vector bundle 
is the intersection of the kernels of $\pi_A$ and $\pi_B$. It has a natural vector 
bundle structure over $M$, the restriction of either structure on $D$, the projection 
of which we call $q_C: C \rightarrow M$. The inclusion $C \hookrightarrow D$ is 
usually denoted by
$$
C_m \ni c \longmapsto \overline{c} \in \pi_A^{-1}(0^A_m) \cap \pi_B^{-1}(0^B_m).
$$

Given a double vector bundle $(D;A,B;M)$, the space of sections $\Gamma_B(D)$ is
generated as a $C^{\infty}(B)$-module by two distinguished classes of
sections (see \cite{Mackenzie11}), the \emph{linear} and the
\emph{core sections} which we now describe.

\begin{definition}
For a section $c\colon  M \rightarrow C$, the corresponding \textbf{core
section} $c^\dagger\colon B \rightarrow D$ is defined as
\begin{equation}\label{core_section}
c^\dagger(b_m) = \tilde{0}_{\vphantom{1}_{b_m}} +_A \overline{c(m)}, \,\, m \in M, \, b_m \in B_m.
\end{equation}
\end{definition}

We denote the corresponding core section $A\to D$ by $c^\dagger$ also, relying on the
argument to distinguish between them. 

\begin{definition}
A section $\xi\in \Gamma_B(D)$ is called \textbf{linear} if $\xi\colon B
\rightarrow D$ is a bundle morphism from $B \rightarrow M$ to $D
\rightarrow A$ over a section $a\in\Gamma(A)$.
\end{definition} 
The space of core sections of $D$ over $B$ will be written
$\Gamma_B^c(D)$ and the space of linear sections $\Gamma^\ell_B(D)$.
Given $\psi\in \Gamma(B^*\otimes C)$, there is a linear section
$\til{\psi}\colon B\to D$ over the zero section $0^A\colon M\to A$ given by
\begin{equation}\label{core_morf}
\widetilde{\psi}(b_m) = \tilo_{b_m}+_A \overline{\psi(b_m)}.
\end{equation}
We call $\widetilde{\psi}$ a \textbf{core-linear section}. 

\begin{example}\label{trivial_dvb}
Let $A, \, B, \, C$ be vector bundles over $M$ and consider
$D=A\times_M B \times_M C$. With the vector bundle structures
$D=q^{!}_A(B\oplus C) \to A$ and $D=q_B^{!}(A\oplus C) \to B$, one
finds that $(D; A, B; M)$ is a double vector bundle called the
\textit{decomposed} or \textit{trivial double vector bundle with 
core $C$}. The core sections are given by
$$
c^\dagger\colon b_m \mapsto (0^A_m, b_m, c(m)), \text{ where } m \in M, \, b_m \in
B_m, \, c \in \Gamma(C),
$$
and similarly for $c^\dagger\colon A\to D$. 
The space of linear sections $\Gamma^\ell_B(D)$ is naturally identified
with $\Gamma(A)\oplus \Gamma(B^*\otimes C)$ via
$$
(a, \psi): b_m \mapsto (a(m), b_m, \psi(b_m)), \text{ where } \psi \in
\Gamma(B^*\otimes C), \, a\in \Gamma(A).
$$

In particular, the fibered product $A\times_M B$ is a double vector
bundle over the sides $A$ and $B$ and has core $M\times 0$.
\end{example}

\subsubsection{Linear splittings and lifts}
\label{subsub:lsl}
A \textbf{linear splitting}\footnote{Note that a linear splitting of
  $D$ is equivalent to a \textbf{decomposition} of $D$, i.e.~an
  isomorphism $\mathbb I\colon A\times_MB\times_MC\to D$ of double
  vector bundles over the identities on the sides and core. Given a
  linear splitting $\Sigma$, the corresponding decomposition $\mathbb
  I$ is given by $\mathbb I(a_m,b_m,c_m)=\Sigma(a_m,b_m)+_B(\tilde
  0_{\vphantom{1}_{b_m}} +_A \overline{c_m})$.  Given a decomposition
  $\mathbb I$, the corresponding linear splitting $\Sigma$ is given by
  $\Sigma(a_m,b_m)=\mathbb I(a_m,b_m,0^C_m)$.} of $(D; A, B; M)$ is an
injective morphism of double vector bundles $\Sigma\colon A\times_M
B\hookrightarrow D$ over the identity on the sides $A$ and $B$.
That every double vector bundle admits local linear
  splittings was proved by \cite{GrRo09}.  Local linear splittings are
  equivalent to double vector bundle charts. Pradines originally
  defined double vector bundles as topological spaces with an atlas of
  double vector bundle charts \cite{Pradines74a}. Using a partition of
  unity, he proved that (provided the double base is a smooth manifold)
  this implies the existence of a global double splitting
  \cite{Pradines77}. Hence, any double vector bundle in the sense of
  Definition \ref{df:dvb} admits a (global) linear splitting.  

A linear splitting $\Sigma$ of $D$ is equivalent to a splitting
$\sigma_A$ of the short exact sequence of $C^\infty(M)$-modules
\begin{equation*}
0 \longrightarrow \Gamma(B^*\otimes C) \hookrightarrow \Gamma^\ell_B(D) 
\longrightarrow \Gamma(A) \longrightarrow 0,
\end{equation*}
where the third map is the map that sends a linear section $(\xi,a)$ to
its base section $a\in\Gamma(A)$.  The splitting $\sigma_A$ will
be called a \textbf{lift}. Given $\Sigma$, the 
lift $\sigma_A\colon \Gamma(A)\to \Gamma_B^\ell(D)$ is given by
$\sigma_A(a)(b_m)=\Sigma(a(m), b_m)$ for all $a\in\Gamma(A)$ and
$b_m\in B$.

In the case of the tangent double of a vector bundle $E\to M$, the
lift from vector fields on $M$ to vector fields on $E$ (see \ref{tangent_double})
would be the horizontal lift corresponding to a connection. We avoid
the word `horizontal' here since `horizontal' and `vertical' refer to
the two structures on $D$.

By the symmetry of a linear splitting, this implies that a 
lift $\sigma_A\colon \Gamma(A)\to\Gamma_B^\ell(D)$ is equivalent to a
lift $\sigma_B\colon \Gamma(B)\to \Gamma_A^\ell(D)$.  Given a
lift $\sigma_A\colon\Gamma(A)\to\Gamma^\ell_B(D)$, the
corresponding lift
$\sigma_B\colon\Gamma(B)\to\Gamma^\ell_A(D)$ is given by
$\sigma_B(b)(a(m))=\sigma_A(a)(b(m))$ for all $a\in\Gamma(A)$,
$b\in\Gamma(B)$.

Note finally that two linear splittings $\Sigma^1,\Sigma^2\colon
A\times_MB\to D$ differ by a section $\Phi_{12}$ of $A^*\otimes
B^*\otimes C\simeq \operatorname{Hom}(A,B^*\otimes C)\simeq
\operatorname{Hom}(B,A^*\otimes C)$ in the following sense.  For each
$a\in\Gamma(A)$ the difference $\sigma_A^1(a)-_B\sigma_A^2(a)$ of
lifts is the core-linear section defined by
$\Phi_{12}(a)\in\Gamma(B^*\otimes C)$. By symmetry,  
$\sigma_B^1(b)-_A\sigma_B^2(b)=\widetilde{\Phi_{12}(b)}$ for each
$b\in\Gamma(B)$.

\subsubsection{The tangent double of a vector bundle}\label{tangent_double}
Let $q_E\colon E\to M$ be a vector bundle.  Then the tangent bundle
$TE$ has two vector bundle structures; one as the tangent bundle of
the manifold $E$, and the second as a vector bundle over $TM$. The
structure maps of $TE\to TM$ are the derivatives of the structure maps
of $E\to M$. The space $TE$ is a double vector bundle with core
bundle $E \to M$.
\begin{equation*}
\begin{xy}
\xymatrix{
TE \ar[d]_{Tq_E}\ar[r]^{p_E}& E\ar[d]^{q_E}\\
 TM\ar[r]_{p_M}& M}
\end{xy}
\end{equation*}

The core vector field corresponding to $e \in \Gamma(E)$ is the 
vertical lift $e^{\uparrow}: E \to TE$, i.e.~the vector field with flow
$\phi\colon E\times \R\to E$, $\phi_t(e'_m)=e'_m+te(m)$. An element of
$\Gamma^\ell_E(TE)=\mx^\ell(E)$ is called a \textbf{linear vector field}. It
is well-known (see e.g.~\cite{Mackenzie05}) that a linear vector field
$\xi\in\mx^l(E)$ covering $X\in\mx(M)$ corresponds to a derivation
$D^*\colon \Gamma(E^*) \to \Gamma(E^*)$ over $X\in \mx(M)$, and hence to 
a derivation $D\colon\Gamma(E)\to\Gamma(E)$ over $X\in \mx(M)$ (the dual
derivation). The precise correspondence is given by\footnote{Since its flow 
is a flow of vector bundle morphisms, a linear vector field sends linear 
functions to linear functions and pullbacks to pullbacks.}
\begin{equation}\label{ableitungen}
\xi(\ell_{\eps}) 
= \ell_{D^*(\eps)} \,\,\,\, \text{ and }  \,\,\, \xi(q_E^*f)= q_E^*(X(f))
\end{equation}
for all $\eps\in\Gamma(E^*)$ and $f\in C^\infty(M)$.  Here
$\ell_\eps$ is the linear function $E\to\R$ corresponding to $\eps$. We will
write $\widehat D$ for the linear vector field corresponding in this
manner to a derivation $D$ of $\Gamma(E)$.  The choice of a linear
splitting $\Sigma$ for $(TE; TM, E; M)$ is equivalent to the choice of
a connection on $E$: Since a linear splitting gives us a linear vector 
field $\sigma_{TM}(X)\in\mx^l(E)$ for each $X$, we can define
$\nabla\colon \mx(M)\times\Gamma(E)\to \Gamma(E)$ by
$\sigma_{TM}(X)=\widehat{\nabla_X}$ for all $X\in\mx(M)$.  Conversely,
a connection $\nabla\colon \mx(M)\times\Gamma(E)\to\Gamma(E)$ defines
a lift $\sigma_{TM}^\nabla$ for $(TE; TM, E; M)$ and a 
linear splitting $\Sigma^\nabla\colon TM\times_M E \to TE$.

We recall as well the relation between the connection and the Lie
bracket of vector fields on $E$.  Given $\nabla$, it is easy to see
using the equalities in \eqref{ableitungen} that, writing $\sigma$ for
$\sigma_{TM}^\nabla$:
\begin{equation}\label{Lie_bracket_VF}
\left[\sigma(X), \sigma(Y)\right]=\sigma[X,Y]-R_\nabla(X,Y)^\uparrow,\qquad
\left[\sigma(X), e^\uparrow\right]=(\nabla_Xe)^\uparrow,\qquad
\left[e^\uparrow,e'^\uparrow\right]=0,
\end{equation}
for all $X,Y\in\mx(M)$ and $e,e'\in\Gamma(E)$.  That is, the Lie
bracket of vector fields on $M$ and the connection encode completely
the Lie bracket of vector fields on $E$.

\medskip

Now let us have a quick look at the other structure on the double vector
bundle $TE$. The lift
$\sigma_{E}^\nabla\colon\Gamma(E)\to\Gamma_{TM}^\ell(TE)$ is given by
\begin{equation}
\sigma_{E}^\nabla(e)(v) = Te(v) +_{TM} (T0^E(v) -_E \overline{\nabla_v e}), \,\, v \in TM, \, e \in \Gamma(E).
\end{equation}

\subsubsection{Dualization and lifts}
Recall that double vector
bundles can be dualized in two distinct ways.
We denote the dual of $D$ as a vector bundle over $A$ by $D\duer A$ and likewise
for $D\duer B$. The dual $D\duer A$ is itself a double 
 vector bundle, with side bundles $A$ and $C^*$ and core $B^*$ \cite{Mackenzie99,Mackenzie11}.
$$ 
{\xymatrix{
D\ar[r]^{\pi_B}\ar[d]_{\pi_A}&   B\ar[d]^{q_{B}}\\
A\ar[r]_{q_A}                   &  M\\
}}
\qquad\qquad
{\xymatrix{
D\duer A \ar[r]^{\pi_{C^*}}\ar[d]_{\pi_A}&   C^*\ar[d]^{q_{C^*}}\\
A\ar[r]_{q_{A}}                   &  M\\
}}
\qquad\qquad
{\xymatrix{
D\duer B \ar[r]^{\pi_B}\ar[d]_{\pi_{C^*}}&   B\ar[d]^{q_B}\\
C^*\ar[r]_{q_{C^*}}                   &  M\\
}}
$$ 
Note also that by dualizing again $D\duer B$ over $C^*$, we get
\[\xymatrix{
D\duer B\duer C^* \ar[r]^{\pi_A}\ar[d]_{\pi_{C^*}}&   A\ar[d]^{q_{A}}  \\
C^*\ar[r]_{q_{C^*}}                   &  M,\\
}\]
with core $B^*$. In the same manner, we get a double vector bundle
$D\duer A\duer C^*$ with sides $B$ and $C^*$ and core $A^*$.

Recall first of all that the vector bundles $D\duer B\to C^*$ and
$D\duer A\to C^*$ are, up to a sign, naturally in duality to each
other \cite{Mackenzie05}. The pairing
\[ \nsp{\cdot\,}{\cdot} \colon (D\duer A)\times_{C^*} (D\duer B)\to \mathbb R
\]  
is defined as follows: 
for $\Phi\in D\duer A$ and $\Psi\in D\duer B$ projecting to the same element
$\gamma_m$ in $C^*$, choose $d\in D$ with $\pi_A(d)=\pi_A(\Phi)$ and
$\pi_B(d)=\pi_B(\Psi)$.  Then $\langle \Phi, d\rangle_A-\langle \Psi,d\rangle_B$ 
does not depend on the choice of $d$ and we set
$\nsp{\Phi}{\Psi}=\langle \Phi, d\rangle_A-\langle \Psi,d\rangle_B$.
 
This implies in particular that $D\duer A$ is canonically (up to a
sign) isomorphic to $D\duer B\duer C^*$ and $D\duer B$ is isomorphic
to $D\duer A\duer C^*$. We will use this below.

\medskip

Each linear section $\xi\in\Gamma_B(D)$ over $a\in\Gamma(A)$ induces a
linear section $\xi^\sqcap\in \Gamma_{C^*}^\ell(D\duer B\duer C^*)$ over
$a$. Namely $\xi$ induces a function $\ell_\xi\colon D\duer B\to\R$
which is fibrewise-linear over $B$ and, using the definition of the 
addition in $D\duer B\to C^*$ (\cite[Equation (7)]{Mackenzie11}, it follows
that $\ell_\xi$ is also linear over $C^*$. The corresponding section 
of $D\duer B\duer C^*\to C^*$ is denoted $\xi^\sqcap$ \cite{Mackenzie11}. 
Thus
\begin{equation}
\label{eq:xisqcap-new}
\nsp{\xi^\sqcap(\gamma)}{\Psi} % _{C^*}
= \ell_\xi(\Psi)
= \langle\Psi,\, \xi(b)\rangle_B
\end{equation}
for $\Psi\in D\duer B$ such that $\pi_B(\Psi)=b$ and $\pi_{C^*}(\Psi)=\gamma$. 

Given a linear splitting $\Sigma\colon A\times_M B\to D$ of $D$, we
get hence a linear splitting 
\newline
$\Sigma^{\star,B}\colon C^*\times_M A\to
D\duer B\duer C^*$, defined by the corresponding lift
$\sigma_{A}^{\star,B}\colon \Gamma(A)\to\Gamma_{C^*}^\ell(D\duer B\duer C^*)$:
\[
\sigma_{A}^{\star,B}(a)=(\sigma_A(a))^\sqcap
\]
for all $a\in\Gamma(A)$.

We now use the (canonical up to a sign) isomorphism of $D\duer A$ with
$D\duer B\duer C^*$ to construct a canonical linear splitting of
$D\duer A$ given a linear splitting of $D$.  We identify $D\duer A$ with 
$D\duer B\duer C^*$ using $-\nsp{\cdot\,}{\cdot}$. Thus we define the
lift $\sigma^\star_A\colon \Gamma(A)\to\Gamma_{C^*}^\ell(D\duer A)$ by 
\begin{equation}\label{iso_sign}
\nsp{\sigma_A^\star(a)}{\cdot} = -\sigma_A^{\star,B}(a)
\end{equation}
for all $a\in\Gamma(A)$.  Note that by \eqref{eq:xisqcap-new}, this implies that 
$\langle \sigma_A^\star(a)(\gamma), \sigma_A(a)(b)\rangle_A=0$
for all $\gamma \in C^*$ and $b\in B$.
The choice of sign in \eqref{iso_sign} is necessary for
$\sigma_A^\star(a)$ to be a linear section of $D\duer A$ over $a$.  To
be more explicit, check or recall from \cite[Equation (28), p.352]{Mackenzie05}
that $\nsp{\sigma_A^\star(a)}{\alpha^\dagger} =- \langle \alpha, a\rangle$
for all $\alpha\in\Gamma(A^*)$ (and $\alpha^\dagger$ the corresponding core
section of $D\duer B$ over $C^*$). But $\langle \sigma_A^{\star,B}(a),
\alpha^\dagger\rangle_{C^*}=q_{C^*}^*\langle a, \alpha\rangle$ by definition
of the pairing of  $D\duer B\duer C^*$ with $D\duer B$. Hence, without the
choice of sign that we make, $\sigma_A^\star(a)$ would be linear over $-a$, hence 
\emph{not} a lift.

By (skew-)symmetry, given the lift
$\sigma_B\colon\Gamma(B)\to \Gamma^\ell_B(D)$, we identify $D\duer B$
with $D\duer A\duer C^*$ using $\nsp{\cdot}{\cdot}$ and
define the lift $\sigma^\star_B\colon
\Gamma(B)\to\Gamma_{C^*}^\ell(D\duer B)$ by $\nsp{\sigma_B^\star(b)}{\cdot} 
= \sigma_B^{\star,A}(b)$ for all $b\in\Gamma(B)$.  (This time, we do not need 
the minus sign.)  As a summary, we have the equations:
\begin{equation}
\nsp{\sigma_A^\star(a)}{\sigma_B^\star(b)}=0, \qquad
\nsp{\sigma_A^\star(a)}{\alpha^\dagger} = -q_{C^*}^*\langle\alpha, a\rangle,\qquad
\nsp{\beta^\dagger}{\sigma_B^\star(b)} = q_{C^*}^*\langle\beta, b\rangle,
\end{equation}
for all $a\in\Gamma(A)$, $b\in\Gamma(B)$, $\alpha\in\Gamma(A^*)$ and $\beta\in\Gamma(B^*)$.

\subsection{VB-algebroids and double Lie algebroids}
\label{subsect:VBa}

What we are here calling VB-algebroids were defined in
\cite{Mackenzie98x,Mackenzie11} and called \LAvbs.  \footnote{The terminology
`\LAvb' followed that of \LAgpds, which were defined in \cite[\S4]{Mackenzie92}, 
on the model of Pradines' \cite{Pradines88} concept of \VBgpd. In \cite{Mackenzie98x} 
and \cite[3.3]{Mackenzie11} \LAvbs were seen as a special case of \LAgpds. 
The terminology `VB-algebroid' of \cite{GrMe10a} distingushes the equivalent
formulation in terms of bracket conditions on the linear and core sections.}

Let $(D; A, B; M)$ be a double vector bundle 
$$
\xymatrix{
D \ar[r]^{\pi_B}\ar[d]_{\pi_A}&   B\ar[d]^{q_B}\\
A\ar[r]_{q_A}                   &  M\\
}
$$ with core $C$.

Then $(D;A,B;M)$ is a VB-algebroid (\cite{Mackenzie98x}; see also
\cite{GrMe10a}) if there are Lie algebroid structures on $D\to B$ and
$A\to M$, such that the vector bundle operations in $D\to A$ are Lie
algebroid morphisms over the corresponding operations in $B\to M$. The
anchor $\Theta_B\colon D\to TB$ is a morphism of double vector bundles
and we denote the induced map of the cores by $\partial\colon C\to B$.
Equivalently \cite{GrMe10a}, $(D \to B; A \to M)$ is a VB-algebroid if
$D \to B$ is a Lie algebroid, the anchor $\Theta_D\colon D \to TB$ is a
bundle morphism over $\rho_A: A \to TM$ and the Lie bracket is linear:
\begin{equation*} 
[\Gamma^\ell_B(D), \Gamma^\ell_B(D)] \subset \Gamma^\ell_B(D), \qquad 
[\Gamma^\ell_B(D), \Gamma^c_B(D)] \subset \Gamma^c_B(D), \qquad 
[\Gamma^c_B(D), \Gamma^c_B(D)]= 0.
\end{equation*}
The vector bundle $A\to M$ is then also a Lie algebroid, with anchor
$\rho_A$ and bracket defined as follows: if $\xi_1,
\xi_2\in\Gamma^\ell_B(D)$ are linear over $a_1,a_2\in\Gamma(A)$, then
the bracket $[\xi_1,\xi_2]$ is linear over $[a_1,a_2]$.

\begin{example}\label{td}
The tangent double $(TE;E,TM;M)$ has a VB-algebroid structure $(TE\to E, TM\to M)$.
\end{example}

\medskip

If $D$ is a VB-algebroid with Lie algebroid structures on $D\to B$ and
$A\to M$ the dual vector bundle $D\duer B\to B$ has a
\emph{Lie-Poisson structure} (a linear Poisson structure), and the
structure on $D\duer B$ is also Lie-Poisson with respect to $D\duer
B\to C^*$ \cite[3.4]{Mackenzie11}. Dualizing this bundle gives a Lie
algebroid structure on $D\duer B\duer C^*\to C^*$. This equips the
double vector bundle $(D\duer B\duer C^*; C^*,A;M)$ with a
VB-algebroid structure. Using the isomorphism defined by
$-\nsp{\cdot}{\cdot}$, the double vector bundle $(D\duer
A;A,C^*;M)$ also has a VB-algebroid structure.

\begin{definition}[\cite{Mackenzie11}]
A \textbf{double Lie algebroid} is a double vector 
bundle $(D;A,B;M)$ with core denoted $C$, and 
with Lie algebroid structures on each of $A\to M$,
$B\to M$, $D\to A$ and $D\to B$ such that each pair of parallel Lie
algebroids gives $D$ the structure of a VB-algebroid, and such that
$(D\duer A\duer C^*, D\duer B\duer C^*)$ with the induced Lie
algebroid structures on base $C^*$ as defined above, is a Lie 
bialgebroid. 
\end{definition}

Equivalently, $D$ is a double Lie algebroid if the pair $(D\duer A, D\duer B)$ 
with the induced Lie algebroid structures on base $C^*$ and the pairing
$\nsp{\cdot}{\cdot}$, is a Lie bialgebroid. One aim of this paper
is to reformulate this definition in terms of specific classes of
sections, so as to allow the user to bypass frequent use of the 
duality of doubles; see Theorem \ref{main}. 

\subsection{Representations up to homotopy and VB-algebroids}
\label{subsect:ruths}
Let $A\to M$ be a Lie algebroid and consider an $A$-connection 
$\nabla$ on a vector bundle $E\to M$.
Then the space $\Omega^\bullet(A,E)$ of $E$-valued Lie algebroid forms has an induced 
operator $\dr_\nabla$ given by the Koszul formula:
\begin{equation*}
\begin{split}
\dr_\nabla\omega(a_1,\ldots,a_{k+1})=&\sum_{i<j}(-1)^{i+j}\omega([a_i,a_j],a_1,\ldots,\hat a_i,\ldots,\hat a_j,\ldots, a_{k+1})\\
&\qquad +\sum_i(-1)^{i+1}\nabla_{a_i}(\omega(a_1,\ldots,\hat a_i,\ldots,a_{k+1}))
\end{split}
\end{equation*}
for all $\omega\in\Omega^k(A,E)$ and $a_1,\ldots,a_{k+1}\in\Gamma(A)$.

Let now $\mathcal E= \bigoplus_{k\in \mathbb{Z}} E_k[k]$ be a 
graded vector bundle. Consider the space $\Omega(A,\mathcal E)$ with grading given by
$$
\Omega(A,\mathcal E)^k = \bigoplus_{i+j=k}\Omega^i(A, E_j).
$$

\begin{definition}\cite{ArCr12}
  A \emph{representation up to homotopy of $A$ on $\mathcal E$} is a map
  $\mathcal D\colon \Omega(A, \mathcal E) \to \Omega(A,\mathcal E)$ with total degree
  $1$ and such that $\mathcal D^2=0$ and
$$
\mathcal D(\alpha \wedge \omega) = \dr_A\alpha \wedge \omega +
(-1)^{|\alpha|} \alpha \wedge \mathcal D(\omega), \text{ for } \alpha
\in \Gamma(\wedge A^*), \, \omega \in \Omega(A,\mathcal E),
$$
where $\dr_A\colon \Gamma(\wedge A^*) \to \Gamma(\wedge A^*)$ is the
Lie algebroid differential.
\end{definition}
Note that Gracia-Saz and Mehta
\cite{GrMe10a} defined this concept 
independently and called them
``superrepresentations''. 

Let $A$ be a Lie algebroid. The representations up to homotopy which
we will consider are always on graded vector bundles $\mathcal E=
E_0\oplus E_1$ concentrated on degrees 0 and 1, so
called \emph{$2$-term graded vector bundles}.  These representations
are equivalent to the following data (see \cite{ArCr12,GrMe10a}):
\begin{enumerate}
\item [(1)] a map $\partial\colon E_0\to E_1$,
\item [(2)] two $A$-connections, $\nabla^0$ and $\nabla^1$ on $E_0$
  and $E_1$, respectively, such that $\partial \circ \nabla^0 =
  \nabla^1 \circ \partial$, \item [(3)] an element $R \in \Omega^2(A,
  \Hom(E_1, E_0))$ such that $R_{\nabla^0} = R\circ \partial$,
  $R_{\nabla^1}=\partial \circ R$ and $\dr_{\nabla^{\Hom}}\omega=0$,
  where $\nabla^{\Hom}$ is the connection induced on $\Hom(E_1,E_0)$
  by $\nabla^0$ and $\nabla^1$.
\end{enumerate}
We will call such a 2-term representation up to homotopy a
\textbf{2-representation} for brevity.
\medskip

Let $(D \to B; A \to M)$ be a VB-algebroid. Then since the
  anchor $\rho_D$ is linear, it sends a core section $c^\dagger$,
  $c\in\Gamma(C)$ to a vertical vector field on $B$.  This defines the
  \textbf{core-anchor} $\partial_B\colon C\to B$ given by,
  $\rho_D(c^\dagger)=(\partial_Bc)^\uparrow$ for all $c\in\Gamma(C)$
  (see \cite{Mackenzie92}).

Choose further a linear splitting $\Sigma\colon A\times_MB\to
D$. Since the anchor of a linear section is linear, for each $a\in
\Gamma(A)$ the vector field $\rho_D(\sigma_A(a))$ defines a derivation
of $\Gamma(B)$ with symbol $\rho(a)$ (see \S
\ref{tangent_double}). This defines a linear connection
$\nabla^{AB}\colon \Gamma(A)\times\Gamma(B)\to\Gamma(B)$:
\[\rho_D(\sigma_A(a))=\widehat{\nabla_a^{AB}}\]
for all $a\in\Gamma(A)$.  Since the bracket of a linear section with a core
section is again linear, we find a linear connection
$\nabla^{AC}\colon\Gamma(A)\times\Gamma(C)\to\Gamma(C)$ such
that \[[\sigma_A(a),c^\dagger]=(\nabla_a^{AC}c)^\dagger\] for all
$c\in\Gamma(C)$ and $a\in\Gamma(A)$.  The difference
$\sigma_A[a_1,a_2]-[\sigma_A(a_1), \sigma_A(a_2)]$ is a core-linear
section for all $a_1,a_2\in\Gamma(A)$.  This defines a vector valued
Lie algebroid form $R\in\Omega^2(A,\operatorname{Hom}(B,C))$ such that
\[[\sigma_A(a_1), \sigma_A(a_2)]=\sigma_A[a_1,a_2]-\widetilde{R(a_1,a_2)},
\]
for all  $a_1,a_2\in\Gamma(A)$. See \cite{GrMe10a} for more details on these constructions.

The following theorem is proved in \cite{GrMe10a}.
\begin{theorem}\label{rajan}
  Let $(D \to B; A \to M)$ be a VB-algebroid and choose a linear
  splitting $\Sigma\colon A\times_MB\to D$.  The triple
  $(\nabla^{AB},\nabla^{AC},R)$ defined as above is a $2$-representation of
  $A$ on the complex $\partial_B\colon C\to B$.

 Conversely, let $(D;A,B;M)$ be a double vector bundle
such that $A$ has a
  Lie algebroid structure and choose a
  linear splitting $\Sigma\colon A\times_MB\to D$. Then if $(\nabla^{AB},\nabla^{AC},R)$ is a
  2-representation of $A$ on a complex $\partial_B\colon C\to B$, then
  the four equations above define a VB-algebroid structure on $(D\to
  B; A\to M)$.

\end{theorem}

 The following formulas
for the brackets of linear and core sections with core-linear
  sections will be very useful in the proof of our main theorem. In
the situation of the previous theorem, we have
\begin{equation}\label{bracket_core_linear}
\left[\sigma_A(a),\widetilde\phi\right]=\widetilde{\nabla_a^{\rm Hom}\phi}
\end{equation}
and 
\begin{equation}\label{bracket_core_linear_2}
\left[c^\dagger,\widetilde\phi\right]=(\phi(\partial_Bc))^\dagger
\end{equation}
for all $a\in\Gamma(A)$, $\phi\in\Gamma(\operatorname{Hom}(B,C))$ and
$c\in\Gamma(C)$.
To see this, write $\phi$ as $\sum f_{ij}\cdot\beta_i\cdot c_j$ with
$f_{ij}\in C^\infty(M)$, $\beta_i\in\Gamma(B^*)$ and
$c_j\in\Gamma(C)$.
Then $\widetilde\phi=\sum q_B^*f_{ij}\cdot\ell_{\beta_i}\cdot
c_j^\dagger$ and one can use the formulas in Theorem \ref{rajan} and
the Leibniz rule to compute the brackets with $\sigma_A(a)$ and $c^\dagger$.

Note that (\ref{bracket_core_linear}) and (\ref{bracket_core_linear_2}) can also be
proved by diagrammatic methods.

\begin{example}\label{double_ruth}
  Choose a linear connection
  $\nabla\colon\mx(M)\times\Gamma(E)\to\Gamma(E)$ and consider the
  corresponding linear splitting $\Sigma^\nabla$ of $TE$ as in Section
  \ref{tangent_double}.  The description of the Lie bracket of vector
  fields in \eqref{Lie_bracket_VF} shows that the 2-representation
  induced by $\Sigma^\nabla$ is the 2-representation of $TM$ on
  $\Id_E\colon E\to E$ given by $(\nabla,\nabla,R_\nabla)$.
\end{example}

\begin{remark}\label{change}
If $\Sigma_1,\Sigma_2\colon A\times_M B\to D$ are two linear
splittings of a VB-algebroid $(D\to B, A\to M)$, then the two corresponding 2-representations are related
by the following identities \cite{GrMe10a}.
\[\nabla^{B,2}_a=\nabla^{B,1}_a+\partial_B\circ \Phi_{12}(a), 
\quad \nabla^{C,2}_a=\nabla^{C,1}_a+\Phi_{12}(a)\circ\partial_B\]
and
\[R^2(a_1,a_2)=R^1(a_1,a_2)+(\dr_{\nabla^{\operatorname{Hom}(B,C)}}\Phi_{12})(a_1,a_2)
+\Phi_{12}(a_1)\partial_B\Phi_{12}(a_2)-\Phi_{12}(a_2)\partial_B\Phi_{12}(a_1)
\]
for all $a,a_1,a_2\in\Gamma(A)$.
\end{remark}

\subsubsection{Dualization and $2$-representations}
Let $(D;A,B;M)$ be a VB-algebroid with Lie algebroid structures on
$D\to B$ and $A\to M$.  Let $\Sigma\colon A\times_MB\to D$ be a linear
splitting of $D$ and denote by $(\nabla^B,\nabla^C,R)$ the
2-representation of the Lie algebroid $A$ on $\partial_B\colon C\to
B$. We have seen above that $(D\duer A;A,C^*;M)$ has an induced
VB-algebroid structure, and we have shown that the linear splitting
$\Sigma$ induces a natural linear spliting $\Sigma^\star\colon
A\times_M C^*\to D\duer B$ of $D\duer A$.  The 2-representation of $A$
that is associated to this splitting is then
$({\nabla^C}^*,{\nabla^B}^*,-R^*) $ on the complex $\partial_B^*\colon
B^*\to C^*$. This is easy to verify, and proved in the
appendix\footnote{The construction of the ``dual'' linear splitting of
  $D\duer A$, given a linear splitting of $D$, is done in
  \cite{DrJoOr13} by dualizing the decomposition and taking its
  inverse. The resulting linear splitting of $D\duer A$ is the same.}
of \cite{DrJoOr13}. One only needs to recall for the proof that, by
construction, $\ell_{\sigma_A^\star(a)}$ equals $\ell_{\sigma_A(a)}$
as a function on $D\duer B$.

\subsubsection{The tangent of a Lie algebroid}\label{tangent_double_al}
Let $(A\to M,\rho,[\cdot\,,\cdot])$ be a Lie algebroid. Then the
tangent $TA\to TM$ has a Lie algebroid structure with bracket defined
by $[Ta_1, Ta_2]=T[a_1,a_2]$, $[Ta_1, a_2^\dagger]=[a_1,a_2]^\dagger$
and $[a_1^\dagger, a_2^\dagger]=0$ for all $a_1,a_2\in\Gamma(A)$. The
anchor of $Ta$ is $\widehat{[\rho(a),\cdot]}\in\mx(TM)$ and the anchor
of $a^\dagger$ is $\rho(a)^\uparrow$ for all $a\in\Gamma(A)$.  This
defines a VB-algebroid structure $(TA\to TM; A\to M)$ on
$(TA;TM,A;M)$.

Given a $TM$-connection on $A$, and so a linear splitting
$\Sigma^\nabla$ of $TA$ as in Section \ref{tangent_double}, the
2-representation of $A$ on $\rho\colon A\to TM$ encoding the
VB-algebroid is $(\nabla^{\rm bas},\nabla^{\rm bas}, R_\nabla^{\rm
  bas})$, where the connections are defined by
\begin{equation*}
\begin{split}
\nabla^{\rm bas}&\colon \Gamma(A)\times\mx(M)\to\mx(M), \\
\nabla^{\rm
    bas}_aX&=[\rho(a), X]+\rho(\nabla_Xa)
\end{split}
\end{equation*} 
and 
\begin{equation*}
\begin{split}
  \nabla^{\rm bas}&\colon
  \Gamma(A)\times\Gamma(A)\to\Gamma(A), \\
  \nabla^{\rm bas}_{a_1}a_2&=[a_1,a_2]+\nabla_{\rho(a_2)}a_1,
\end{split}
\end{equation*}  and $R_\nabla^{\rm
    bas}\in \Omega^2(A,\operatorname{Hom}(TM,A))$ is given by
  \[R_\nabla^{\rm
    bas}(a_1,a_2)X=-\nabla_X[a_1,a_2]+[\nabla_Xa_1,a_2]+[a_1,\nabla_Xa_2]+\nabla_{\nabla^{\rm
      bas}_{a_2}X}a_1-\nabla_{\nabla^{\rm bas}_{a_1}X}a_2
\]
for all $X\in\mx(M)$, $a, a_1,a_2\in\Gamma(A)$.

\section{Main theorem and examples}
\label{sect:mte}
We define in this section the notion of \emph{matched pair of
representations up to homotopy} --- as above, we consider
only representations up to homotopy which are concentrated in
degrees $0$ and $1$; that is, $2$-representations. We then state 
our main result: a double vector bundle endowed with two VB-algebroid 
structures (on each of its sides) is a double Lie algebroid if and 
only if, for each linear splitting, the two induced representations 
up to homotopy form a matched pair. 

In the second part of this section, we work out the example of the tangent 
double of a Lie algebroid.

\subsection{Matched pairs of representations up to homotopy and main result}
\begin{definition}\label{matched}
  Let $(A\to M, \rho_A, [\cdot\,,\cdot])$ and and $(B\to M, \rho_B,
  [\cdot\,,\cdot])$ be two Lie algebroids and assume that $A$ acts on
  $\partial_B\colon C\to B$ up to homotopy via
  $(\nabla^{AB},\nabla^{AC}, R_{A})$ and $B$ acts on
  $\partial_A\colon C\to A$ up to homotopy via
  $(\nabla^{BA},\nabla^{BC}, R_{B})$.  Then we say that the two $2$-representations form 
a \emph{matched pair} if the following hold:\footnote{For the sake of
    simplicity, from now on we usually write $\nabla$ for all four
    connections. It is always clear from the indexes which
    connection is meant. We write $\nabla^A$ for the $A$-connection
    induced by $\nabla^{AB}$ and $\nabla^{AC}$ on $\wedge^2 B^*\otimes
    C$ and $\nabla^B$ for the $B$-connection induced on $\wedge^2
    A^*\otimes C$.}
\begin{enumerate}
\item[(M1)] $\rho_A\circ \partial_A=\rho_B\circ\partial_B$,
\item[(M2)]
  $\nabla_{\partial_Ac_1}c_2-\nabla_{\partial_Bc_2}c_1=-\nabla_{\partial_Ac_2}c_1+\nabla_{\partial_Bc_1}c_2$,
\item[(M3)] $[a,\partial_Ac]=\partial_A(\nabla_ac)-\nabla_{\partial_Bc}a$,
\item[(M4)] $[b,\partial_Bc]=\partial_B(\nabla_bc)-\nabla_{\partial_Ac}b$,
\item[(M5)] $[\rho_A(a),\rho_B(b)]=\rho_B(\nabla_ab)-\rho_A(\nabla_ba)$,
\item[(M6)]
$\nabla_b\nabla_ac-\nabla_a\nabla_bc-\nabla_{\nabla_ba}c+\nabla_{\nabla_ab}c=
R_{B}(b,\partial_Bc)a-R_{A}(a,\partial_Ac)b$,
\item[(M7)]
  $\partial_A(R_{A}(a_1,a_2)b)=-\nabla_b[a_1,a_2]+[\nabla_ba_1,a_2]+[a_1,\nabla_ba_2]+\nabla_{\nabla_{a_2}b}a_1-\nabla_{\nabla_{a_1}b}a_2$,
\item[(M8)]
  $\partial_B(R_{B}(b_1,b_2)a)=-\nabla_a[b_1,b_2]+[\nabla_ab_1,b_2]+[b_1,\nabla_ab_2]+\nabla_{\nabla_{b_2}a}b_1-\nabla_{\nabla_{b_1}a}b_2$,
\end{enumerate}
for all $a,a_1,a_2\in\Gamma(A)$, $b,b_1,b_2\in\Gamma(B)$ and
$c,c_1,c_2\in\Gamma(C)$, and
\begin{enumerate}
\item[(M9)] $\dr_{\nabla^A}R_{B}=\dr_{\nabla^B}R_{A}\in \Omega^2(A,
  \wedge^2B^*\otimes C)=\Omega^2(B,\wedge^2 A^*\otimes C)$, where
  $R_{B}$ is seen as an element of $\Omega^1(A, \wedge^2B^*\otimes
  C)$ and $R_{A}$ as an element of $\Omega^1(B, \wedge^2A^*\otimes
  C)$.
\end{enumerate}
\end{definition}
\begin{remark}
\begin{enumerate}
\item Compare equations (M1) to (M9) with equations (50) to (58) of 
\cite{Voronov12}.
\item Note that if $C$ is trivial, then $\partial_A,\partial_B$, $R_A, R_B$ and
    $\nabla^{AC}, \nabla^{BC}$ are trivial. In that case, equations
    (M1)--(M4), (M6) and (M9) and the left hand sides of (M7) and (M8)
    vanish. We find hence the definition of a matched pair of
    representations of Lie algebroids \cite{Mokri97,Lu97}.
\end{enumerate}
\end{remark}

\begin{remark}
\label{rmk:C}
Note that the vector bundle $C$ inherits a Lie algebroid structure
with anchor $\rho_C:=\rho_A\circ \partial_A=\rho_B\circ\partial_B$ and with
bracket given by $[c_1, c_2]=\nabla_{\partial_Ac_1}c_2-\nabla_{\partial_Bc_2}c_1$ 
for all $c_1,c_2\in\Gamma(C)$. 

The choice of sign is the natural one for the Leibniz identity to be satisfied.
The proof of the Jacobi identity is not completely
straightforward; it follows from (M3), (M4) and (M6).
A detailed proof of a more general result, but with the same type of
computation, is given in \cite[Theorem 6.12]{Jotz14b}.
Note that (M3) together with $\partial_A \circ \nabla^{BC} = \nabla^{BA} \circ \partial_A$ 
(Equation (2) in the definition of a $2$-representation) shows that 
$\partial_A \co C \to A$ is a Lie algebroid morphism. In
the same manner, (M4) together with $\partial_B \circ \nabla^{AC} = 
\nabla^{AB} \circ \partial_B$ shows that $\partial_B \co C \to B$ is a Lie algebroid morphism.
\end{remark}

Theorem \ref{main} is our main result. The proof is in \S\ref{sect:proof}. 

\begin{theorem}\label{main}
Let $(D;A,B;M)$ be a double vector bundle with VB-algebroid structures 
on both $(D\to A, B\to M)$ and $(D\to B, A\to M)$.  Choose a linear
splitting $\Sigma$ of $D$ and let $\mathcal D_A$ and $\mathcal D_B$
be the two $2$-representations defined by the lifts
$\sigma_A$ and $\sigma_B$. Then $(D;A,B;M)$ is a double Lie
algebroid if and only if the two $2$-representations form a matched
pair.
\end{theorem}

It is easy to see using Remark \ref{change}
that the induced Lie algebroid structure on the core $C$ of the double Lie algebroid does not depend on the choice of splitting.

\begin{remark}
\label{rmk:mp}
Given a matched pair of representations of Lie algebroids $A$ and $B$
on the same base $M$, the direct sum vector bundle $A\oplus B$ has a Lie algebroid 
structure, the \emph{bicrossproduct Lie algebroid}, denoted $A\bowtie B$  
\cite{Lu97,Mokri97}. The matched pair structure also induces on the decomposed
double vector bundle $A\times_M B$ a double Lie algebroid structure and, conversely,
any vacant double Lie algebroid (that is, a double Lie algebroid for which the core
is zero) arises from a matched pair of Lie algebroids in this 
way \cite[\S6]{Mackenzie11}.
\end{remark}

\subsection{The tangent double of a Lie algebroid}
Let $A\to M$ be a Lie algebroid with anchor $\rho$. We have seen in
Section \ref{tangent_double_al} that
 \begin{equation*}
\begin{xy}
\xymatrix{
TA \ar[d]_{Tq_A}\ar[r]^{p_A}& A\ar[d]^{q_A}\\
 TM\ar[r]_{p_M}& M}
\end{xy}
\end{equation*}
is endowed with two VB-algebroid structures. $(TA\to A,TM\to M)$ has
the standard tangent bundle Lie algebroid structure (Example \ref{td})
and $(TA\to TM, A\to M)$ is the tangent prolongation of $A\to M$
(Section \ref{tangent_double_al}).

Recall from Section \ref{tangent_double} that a linear connection
$\nabla\colon \mx(M)\times\Gamma(A)\to\Gamma(A)$ defines a linear
splitting $\Sigma^\nabla\colon A\times_M TM\to TA$.  This linear
splitting induces the two following 2-representations:
\begin{enumerate}
\item The VB-algebroid $(TA\to A,TM\to M)$ is described by the
  2-representation of $TM$ on $\Id_A\colon A\to A$ via
  $(\nabla,\nabla,R_\nabla)$ (Example \ref{double_ruth}). The anchor
  of $TM$ is $\Id_{TM}$ and the bracket is the Lie bracket of vector
  fields.
\item The VB-algebroid $(TA\to TM, A\to TM)$ is described by the
  2-representation of $A$ on $\rho\colon A\to TM$ via $(\nabla^{\rm
    bas},\nabla^{\rm bas},R_\nabla^{\rm bas})$ (Section
  \ref{tangent_double_al}).
\end{enumerate}

We check that these two 2-representations form a matched pair. This
will provide a new proof of the fact that the tangent double of a Lie
algebroid is a double Lie algebroid \cite{Mackenzie11}.  Condition (M1) in Definition
\ref{matched} is immediate, (M2) and (M3) are just two times the
definition of $\nabla^{\rm bas}\colon
\Gamma(A)\times\Gamma(A)\to\Gamma(A)$ and (M4) and (M5) are two times
the definition of $\nabla^{\rm bas}\colon
\Gamma(A)\times\mx(M)\to\mx(M)$. Condition (M7) is the definition of 
$R^{\rm bas}_\nabla$.  Hence, we only have to check (M6), (M8) and (M9).

In the following, $X,X_1,X_2$ will be arbitrary vector fields on $M$
and $a,a_1,a_2$ arbitrary sections of $A$.

\noindent\textbf{(M6)} The left-hand side of (M6) is  
\begin{equation*}
\begin{split}
\nabla_{X}\nabla^{\rm bas}_{a_1}a_2-\nabla^{\rm bas}_{a_1}\nabla_Xa_2 &
-\nabla^{\rm bas}_{\nabla_Xa_1}a_2+\nabla_{\nabla^{\rm bas}_{a_1}X}a_2\\
& = \nabla_{X}[a_1,a_2]+\nabla_X\nabla_{\rho(a_2)}a_1
-[a_1,\nabla_Xa_2]-\nabla_{\rho(\nabla_Xa_2)}a_1\\
& \relphantom{=}-[\nabla_Xa_1,a_2]-\nabla_{\rho(a_2)}\nabla_Xa_1
+\nabla_{[\rho(a_1),X]}a_2+\nabla_{\rho(\nabla_Xa_1)}a_2.
\end{split}
\end{equation*}
The second and sixth term add up to
$R(X,\rho(a_2))a_1+\nabla_{[X,\rho(a_2)]}a_1$ and the first, third,
and fifth term to $-R_\nabla^{\rm bas}(a_1,a_2)X+\nabla_{\nabla^{\rm
    bas}_{a_2}X}a_1-\nabla_{\nabla^{\rm bas}_{a_1}X}a_2$. The
definition of $\nabla^{\rm bas}\colon \Gamma(A)\times\mx(M)\to\mx(M)$
yields then immediately the right hand side of (M6), namely
\[R(X,\rho(a_2))a_1-R_\nabla^{\rm bas}(a_1,a_2)X.\] 

\medskip

\noindent\textbf{(M8)} This equation is easily verified:
\begin{equation*}
\begin{split}
-\nabla^{\rm bas}_a[X_1,X_2] + & [\nabla^{\rm bas}_aX_1,X_2]+[X_1,\nabla^{\rm bas}_aX_2] +\nabla^{\rm bas}_{\nabla_{X_2}a}X_1-\nabla^{\rm bas}_{\nabla_{X_1}a}X_2\\
& =-[\rho(a),[X_1,X_2]]-\rho(\nabla_{[X_1,X_2]}a)
+[[\rho(a),X_1] +\rho(\nabla_{X_1}a),X_2]\\
&\relphantom{=}+[X_1,[\rho(a),X_2]+\rho(\nabla_{X_2}a)] +[\rho(\nabla_{X_2}a),X_1]\\
&\relphantom{=} +\rho(\nabla_{X_1}\nabla_{X_2}a)
-[\rho(\nabla_{X_1}a),X_2]-\rho(\nabla_{X_2}\nabla_{X_1}a)\\
& =\rho(R_\nabla(X_1,X_2)a)
\end{split}
\end{equation*}
To get the second equality, we use the Jacobi identity for the Lie bracket 
of vector fields. (The four remaining terms cancel pairwise). 

\medskip

\noindent\textbf{(M9)} As one would expect, checking (M9) is a long, but
straightforward computation. We carry this out in detail here, because 
we will not give all the details of the proof of Theorem \ref{main}. 
We begin by computing
\begin{align*}
(\dr_{\nabla}R_\nabla^{\rm bas})(X_1,X_2)(a_1,a_2)=&-R_\nabla^{\rm bas}(a_1,a_2)[X_1,X_2]\\
&+\nabla_{X_1}(R_\nabla^{\rm
  bas}(a_1,a_2)X_2)-\nabla_{X_2}(R_\nabla^{\rm
  bas}(a_1,a_2)X_1)\\
&-R_\nabla^{\rm bas}(\nabla_{X_1}a_1,a_2)X_2 -R_\nabla^{\rm
  bas}(a_1,\nabla_{X_1}a_2)X_2\\
&+R_\nabla^{\rm bas}(\nabla_{X_2}a_1,a_2)X_1 +R_\nabla^{\rm bas}(a_1,\nabla_{X_2}a_2)X_1.
\end{align*}
This expands out to
\begin{align*}
&\nabla_{[X_1,X_2]}[a_1,a_2]-[\nabla_{[X_1,X_2]} a_1,a_2]
  -[a_1,\nabla_{[X_1,X_2]} a_2]+\nabla_{\nabla^{\rm
      bas}_{a_1}{[X_1,X_2]}}a_2-\nabla_{\nabla^{\rm
      bas}_{a_2}{[X_1,X_2]}}a_1\\
&+\nabla_{X_1}\left(-\nabla_{X_2}[a_1,a_2]+\cancel{[\nabla_{X_2} a_1,a_2]}
  +\cancel{[a_1,\nabla_{X_2} a_2]}-\nabla_{\nabla^{\rm
      bas}_{a_1}{X_2}}a_2+\nabla_{\nabla^{\rm
      bas}_{a_2}{X_2}}a_1\right)\\
&-\nabla_{X_2}\left(-\nabla_{X_1}[a_1,a_2]+\cancel{[\nabla_{X_1} a_1,a_2]}
  +\cancel{[a_1,\nabla_{X_1} a_2]}-\nabla_{\nabla^{\rm
      bas}_{a_1}{X_1}}a_2+\nabla_{\nabla^{\rm
      bas}_{a_2}{X_1}}a_1
\right)\\
&+\cancel{\nabla_{X_2}[\nabla_{X_1}a_1,a_2]}-[\nabla_{X_2} \nabla_{X_1}a_1,a_2]
  -\cancel{[\nabla_{X_1}a_1,\nabla_{X_2} a_2]}+\nabla_{\nabla^{\rm
      bas}_{\nabla_{X_1}a_1}{X_2}}a_2-\nabla_{\nabla^{\rm
      bas}_{a_2}{X_2}}\nabla_{X_1}a_1\\
&+\cancel{\nabla_{X_2}[a_1,\nabla_{X_1}a_2]}-\cancel{[\nabla_{X_2} a_1,\nabla_{X_1}a_2]}
  -[a_1,\nabla_{X_2} \nabla_{X_1}a_2]+\nabla_{\nabla^{\rm
      bas}_{a_1}{X_2}}\nabla_{X_1}a_2-\nabla_{\nabla^{\rm
      bas}_{\nabla_{X_1}a_2}{X_2}}a_1\\
&-\cancel{\nabla_{X_1}[\nabla_{X_2}a_1,a_2]}+[\nabla_{X_1} \nabla_{X_2}a_1,a_2]
  +\cancel{[\nabla_{X_2}a_1,\nabla_{X_1} a_2]}-\nabla_{\nabla^{\rm
      bas}_{\nabla_{X_2}a_1}{X_1}}a_2+\nabla_{\nabla^{\rm
      bas}_{a_2}{X_1}}\nabla_{X_2}a_1\\
&-\cancel{\nabla_{X_1}[a_1,\nabla_{X_2}a_2]}+\cancel{[\nabla_{X_1} a_1,\nabla_{X_2}a_2]}
  +[a_1,\nabla_{X_1} \nabla_{X_2}a_2]-\nabla_{\nabla^{\rm
      bas}_{a_1}{X_1}}\nabla_{X_2}a_2+\nabla_{\nabla^{\rm
      bas}_{\nabla_{X_2}a_2}{X_1}}a_1.
\end{align*}
Twelve terms of this equation cancel pairwise as shown, and a reordering 
of the remaining terms yields
\begin{align*}
  -R_\nabla&(X_1,X_2)[a_1,a_2]+[R_\nabla(X_1,X_2)a_1,a_2]+[a_1,R_\nabla(X_1,X_2)a_2]\\
  &+R_\nabla(X_2,\nabla^{\rm
    bas}_{a_1}{X_1})a_2+\nabla_{[X_2,\nabla^{\rm bas}_{a_1}X_1]}a_2
  -R_\nabla(X_2,\nabla^{\rm
    bas}_{a_2}{X_1})a_1-\nabla_{[X_2,\nabla^{\rm bas}_{a_2}X_1]}a_1\\
  &-R_\nabla(X_1,\nabla^{\rm
    bas}_{a_1}{X_2})a_2-\nabla_{[X_1,\nabla^{\rm bas}_{a_1}X_2]}a_2
  +R_\nabla(X_1,\nabla^{\rm
    bas}_{a_2}{X_2})a_1+\nabla_{[X_1,\nabla^{\rm bas}_{a_2}X_2]}a_1\\
  & +\nabla_{\nabla^{\rm bas}_{a_1}{[X_1,X_2]}}a_2-\nabla_{\nabla^{\rm
      bas}_{a_2}{[X_1,X_2]}}a_1-\nabla_{\nabla^{\rm
      bas}_{\nabla_{X_2}a_1}{X_1}}a_2\\
  &+\nabla_{\nabla^{\rm bas}_{\nabla_{X_2}a_2}{X_1}}a_1
  +\nabla_{\nabla^{\rm
      bas}_{\nabla_{X_1}a_1}{X_2}}a_2-\nabla_{\nabla^{\rm
      bas}_{\nabla_{X_1}a_2}{X_2}}a_1.
\end{align*}
By (M8), this equals
\begin{align*}
  -R_\nabla&(X_1,X_2)[a_1,a_2]+[R_\nabla(X_1,X_2)a_1,a_2]+[a_1,R_\nabla(X_1,X_2)a_2]\\
  &+R_\nabla(X_2,\nabla^{\rm bas}_{a_1}{X_1})a_2-R_\nabla(X_2,\nabla^{\rm
    bas}_{a_2}{X_1})a_1
  -R_\nabla(X_1,\nabla^{\rm bas}_{a_1}{X_2})a_2+R_\nabla(X_1,\nabla^{\rm
    bas}_{a_2}{X_2})a_1\\
  &+\nabla_{\rho(R_\nabla(X_1,X_2)a_2)}a_1-\nabla_{\rho(R_\nabla(X_1,X_2)a_1)}a_2,
\end{align*}
which is 
\begin{align*}
-&R_\nabla(X_1,X_2)[a_1,a_2]-\nabla^{\rm
  bas}_{a_2}R_\nabla(X_1,X_2)a_1+\nabla^{\rm bas}_{a_1}R_\nabla(X_1,X_2)a_2\\
&-R_\nabla(\nabla^{\rm bas}_{a_1}{X_1},X_2)a_2+R_\nabla(\nabla^{\rm bas}_{a_2}{X_1},X_2)a_1
-R_\nabla(X_1,\nabla^{\rm bas}_{a_1}{X_2})a_2+R_\nabla(X_1,\nabla^{\rm bas}_{a_2}{X_2})a_1\\
=&(\dr_{\nabla^{\rm bas}}R_\nabla)(a_1,a_2)(X_1,X_2).
\end{align*}

\section{Proof of the theorem}
\label{sect:proof}

We will prove the theorem by checking the Lie bialgebroid condition
only on particular families of sections; the linear sections and the core
sections. The main difficulty is to understand the additional conditions 
which have to be verified by the families of sections for the proof to be 
complete. This is done in Subsection \ref{families}. In Subsection \ref{proof}, 
we will show how the equations found in Subsection \ref{families} imply 
(M1)--(M9) and vice-versa.

\subsection{Families of sections of Lie bialgebroids}\label{families}
We recall the definition of a Lie bialgebroid \cite{MaXu94};
see also \cite[Chapter 12]{Mackenzie05}. We will then show how the equation defining
a Lie bialgebroid $(A,A^*)$ can be verified only on families of spanning sections of $A$ and $A^*$.
 
\begin{definition}
  Let $q_A\colon A\to M$ and $q_{A^*}\colon A^*\to M$ be a pair of
  dual vector bundles, and suppose each has a Lie algebroid structure,
  with anchors $\rho\colon A\to TM$ and $\rho_*\colon A^*\to TM$
  respectively, and brackets $[\cdot\,,\cdot]$ and
  $[\cdot\,,\cdot]_*$.

  Then $(A,A^*)$ is a \emph{Lie bialgebroid} if for all
  $a_1,a_2\in\Gamma(A)$:
\begin{equation}
\label{eq:bialgd}
\dr_{A^*}[a_1,a_2]=[\dr_{A^*}a_1,a_2]+[a_1,\dr_{A^*}a_2].
\end{equation}
\end{definition}

The brackets on the RHS are extensions to $2$-vectors by
standard Schouten calculus. 

It is often very convenient to check this condition only on the
elements of a given set of sections $\mathcal S\subseteq\Gamma(A)$ which
spans $\Gamma(A)$ as a $C^\infty(M)$-module.  We will formalize this
technique shortly. We first need to recall some consequences of the
definition. 

The proof of the following proposition is a straightforward
computation.
\begin{proposition}
  Let $A$ and $A^*$ be dual vector bundles with Lie algebroid
  structures.
For $a_1,a_2\in\Gamma(A)$, $\alpha_1,\alpha_2\in\Gamma(A^*)$ and $f\in
C^\infty(M)$, we have 
\begin{equation}
\label{eq:fa2}
\begin{split}
  &\, (\dr_{A^*}[a_1,fa_2]-[\dr_{A^*}a_1,fa_2]-[a_1,\dr_{A^*}(fa_2)])(\alpha_1,\alpha_2)\\
  =\,&\, f\cdot
  \left(\dr_{A^*}[a_1,a_2]-[\dr_{A^*}a_1,a_2]-[a_1,\dr_{A^*}a_2]
  \right))(\alpha_1,\alpha_2)\\\
  &\, -\langle a_2,\alpha_2\rangle\cdot \left(
    [\rho(a_1),\rho_*(\alpha_1)](f)-\rho_*(\ldr{a_1}\alpha_1)(f)+\rho(\ldr{\alpha_1}a_1)(f)-\rho_*(\dr_{A}f)\langle
    a_1,\alpha_1\rangle
  \right)\\
  &\, +\langle a_2,\alpha_1\rangle\cdot \left(
    [\rho(a_1),\rho_*(\alpha_2)](f)-\rho_*(\ldr{a_1}\alpha_2)(f)+\rho(\ldr{\alpha_2}a_1)(f)-\rho_*(\dr_{A}f)\langle
    a_1,\alpha_2\rangle \right).
\end{split}
\end{equation}
\end{proposition}

Now assume that $(A,A^*)$ is a Lie bialgebroid. Take any $a_1\in\Ga A$ and 
any nonvanishing $\alpha_1\in\Gamma(A^*)$. Choose a nonvanishing $a_2\in\Gamma(A)$ 
and an $\alpha_2\in\Gamma(A^*)$ such that 
$\langle a_2,\alpha_1\rangle=0$ and $\langle a_2,\alpha_2\rangle=1$.
(If $A$ has rank $1$ then (\ref{b2}) below is vacuously true.)
Equation (\ref{eq:fa2}) now reduces to
\begin{equation}\label{b2}
[\rho(a_1),\rho_*(\alpha_1)](f)-\rho_*(\ldr{a_1}\alpha_1)(f)+\rho(\ldr{\alpha_1}a_1)(f)-\rho_*(\dr_{A}f)\langle
a_1,\alpha_1\rangle=0
\end{equation}
for all $a_1\in\Gamma(A)$, $f\in C^\infty(M)$, 
and nonvanishing $\alpha_1\in\Gamma(A^*)$. A straightworward
computation shows that the left-hand side of \eqref{b2} is tensorial
in the term $\alpha_1$. Hence, \eqref{b2} holds for all
$\alpha_1\in\Gamma(A^*)$. (For another proof, see \cite[12.1.8]{Mackenzie05}.)

On the other hand, the left-hand side of \eqref{b2} is not tensorial in the term 
$a_1$. We multiply $a_1$ by a function $g\in C^\infty(M)$ in
this equation, expand out, and subtract 
\[g\cdot
\left([\rho(a_1),\rho_*(\alpha_1)](f)-\rho_*(\ldr{a_1}\alpha_1)(f)+\rho(\ldr{\alpha_1}a_1(f)-\rho_*(\dr_{A}f)\langle
  a_1,\alpha_1\rangle \right)=0.
\]
We get that 
\[\langle a_1, \alpha_1\rangle\cdot \left(  
  -\rho_*(\dr_Ag)(f)-\rho_*(\dr_Af)(g) \right)=0.
\]
Again, since $a_1$ and $\alpha_1$ were arbitrary, we have found
\begin{equation*}
-\rho_*(\dr_Ag)(f)=\rho_*(\dr_Af)(g)\quad \text{for all } \quad f,g\in C^\infty(M),
\end{equation*}
which is easily seen to be equivalent to
\begin{equation}
\label{eq:Poissons}
-\rho\circ \rho_*^*=\rho_*\circ \rho^*,
\end{equation}
see also \cite{MaXu94},\cite[\S12.1]{Mackenzie05}. The map $\rho_*\circ \rho^*\co
T^*M\to TM$ defines a Poisson structure on $M$, which we take to be the \emph{Poisson
structure on $M$ induced by the Lie bialgebroid structure.} 

These considerations lead to the following result.

\begin{proposition}\label{bialgebroid_conditions}
  Let $q_A\colon A\to M$ and $q_{A^*}\colon A^*\to M$ be a pair of
  dual vector bundles, and suppose each has a Lie algebroid structure,
  with anchors $\rho\colon A\to TM$ and $\rho_*\colon A^*\to TM$
  respectively, and brackets $[\cdot\,,\cdot]$ and
  $[\cdot\,,\cdot]_*$.
 Let $\mathcal S$ be a subset of $\Gamma(A)$ which spans $\Gamma(A)$
  as a $C^\infty(M)$-module.

  Then $(A,A^*)$ is a Lie bialgebroid if and only if the following
  three conditions hold.
  \renewcommand{\theenumi}{{{\rm\roman{enumi}}}}
\begin{enumerate}
\item[(B1)] $\dr_{A^*}[a_1,a_2]=[\dr_{A^*}a_1,a_2]+[a_1,\dr_{A^*}a_2]$ for all $a_1,a_2\in
  \mathcal S$,
\item[(B2)]
  $[\rho(a),\rho_*(\alpha)](f)-\rho_*(\ldr{a}\alpha)(f)+\rho(\ldr{\alpha}a)(f)-\rho_*(\dr_{A}f)\langle
  a,\alpha\rangle=0$ for all $a\in \mathcal S$, $\alpha\in\Gamma(A^*)$ and
$f\in C^\infty(M)$, 
  and
\item[(B3)] $-\rho\circ \rho_*^*=\rho_*\circ \rho^*$.
\end{enumerate}
\end{proposition}

\begin{proof}
  We proved above that these three conditions hold when $(A,A^*)$ is a
  Lie bialgebroid. For the converse, a quick computation using
  (B1) and the considerations before the proposition shows that
\begin{align*}
  &\, (\dr_{A^*}[ga_1,fa_2]-[\dr_{A^*}(ga_1),fa_2]-[ga_1,\dr_{A^*}(fa_2)])(\alpha_1,\alpha_2)\\
  =\,&\, fg\cdot
  \left(\dr_{A^*}[a_1,a_2]-[\dr_{A^*}a_1,a_2]-[a_1,\dr_{A^*}a_2]
  \right))(\alpha_1,\alpha_2)\\
  &\, -fg\langle a_2,\alpha_2\rangle\cdot \left(
    [\rho(a_1),\rho_*(\alpha_1)](f)-\rho_*(\ldr{a_1}\alpha_1)(f)+\rho(\ldr{\alpha_1}a_1)(f)-\rho_*(\dr_{A}f)\langle
    a_1,\alpha_1\rangle
  \right)\\
  &\, +fg\langle a_2,\alpha_1\rangle\cdot \left(
    [\rho(a_1),\rho_*(\alpha_2)](f)-\rho_*(\ldr{a_1}\alpha_2)(f)+\rho(\ldr{\alpha_2}a_1)(f)-\rho_*(\dr_{A}f)\langle
    a_1,\alpha_2\rangle
  \right)\\
  &\, +fg\langle a_1,\alpha_1\rangle\cdot \left(
    [\rho(a_2),\rho_*(\alpha_2)](f)-\rho_*(\ldr{a_2}\alpha_2)(f)+\rho(\ldr{\alpha_2}a_2)(f)-\rho_*(\dr_{A}f)\langle
    a_2,\alpha_2\rangle
  \right)\\
  &\, -fg\langle a_1,\alpha_2\rangle\cdot \left(
    [\rho(a_2),\rho_*(\alpha_1)](f)-\rho_*(\ldr{a_2}\alpha_1)(f)+\rho(\ldr{\alpha_1}a_2)(f)-\rho_*(\dr_{A}f)\langle
    a_2,\alpha_1\rangle \right).
\end{align*}
for all $a_1,a_2\in\mathcal S$,
$\alpha_1,\alpha_2\in\Gamma(A^*)$ and $f,g\in C^\infty(M)$.  This vanishes by (B1) and (B2). Since the Lie bialgebroid condition
is additive and $\Gamma(A)$ is spanned as a $C^\infty(M)$-module by $\mathcal S$, we are done.
\end{proof}

\begin{remark}
  In Proposition \ref{bialgebroid_conditions}, the first two
  conditions are $C^\infty(M)$-linear in the $\Gamma(A^*)$-argument,
  so it is sufficient to check them on a subset $\mathcal R\subseteq
  \Gamma(A^*)$ that spans $\Gamma(A^*)$ as a $C^\infty(M)$-module.
\end{remark}

\subsection{The Lie bialgebroid conditions on lifts and on core sections}\label{proof}
We write here $\Theta_A\colon D\duer A\to TC^*$ for the anchor of $D\duer A\to C^*$ and 
$\Theta_B\colon D\duer B\to TC^*$ for the anchor of $D\duer B\to C^*$. 
We set \[\mathcal S:=\Gamma_{C^*}^c(D\duer A)\cup \sigma_A^\star(\Gamma(A))\]
 and 
\[\mathcal R:=\Gamma_{C^*}^c(D\duer B)\cup \sigma_B^\star(\Gamma(B)).
\]

\begin{proposition}
Condition {\rm (B3)} on $\mathcal S$ and $\mathcal R$ is equivalent to 
{\rm (M1)} and {\rm (M2)}.
\end{proposition}

\begin{proof}
Since
\[\Theta_A\circ \Theta_B^*, \Theta_B\circ\Theta_A^*\colon  T^*C^*\to TC^*.
\]
are vector bundle maps, it is sufficient to check (B3)
on $\dr F$ for $F\in C^\infty(C^*)$. In fact, it is even sufficient to
check (B3) on $\dr(q_{C^*}^*f)$ for $f\in C^\infty(M)$ and 
$\dr\ell_c$ for $c\in\Gamma(C)$.

\medskip

 Choose first $f\in C^\infty(M)$ and consider $q_{C^*}^*f\in
 C^\infty(C^*)$.
We have for any section $b\in \Gamma(B)$:
\[\nsp{\Theta_B^*(\dr q_{C^*}^* f)}{\sigma_B^\star(b)}
=\widehat{\nabla_b^*}(q_{C^*}^* f)=q_{C^*}^*(\rho_B(b)f)
\]
and for any $\alpha\in\Gamma(A^*)$:
\[\nsp{\Theta_B^*(\dr q_{C^*}^* f)}{\alpha^\dagger}
=(\partial_A^*\alpha)^\uparrow(q_{C^*}^* f)=0.
\]
This shows
\begin{equation}\label{anchor_on_core}
\Theta_B^*(\dr q_{C^*}^* f)=(\rho_B^*\dr
f)^\dagger\in \Gamma_{C^*}^c(D\duer A).
\end{equation} We get consequently 
$\Theta_A\circ\Theta_B^*(\dr q_{C^*}^* f)=(\partial_B^*\rho_B^*\dr
f)^\uparrow\in \mx(C^*)$.
In the same manner, we find $\Theta_A\circ\Theta_B^*(\dr q_{C^*}^* f)=(-\partial_A^*\rho_A^*\dr
f)^\uparrow\in \mx(C^*)$. The
equality of $\Theta_A\circ \Theta_B^*$ and  $-\Theta_B\circ\Theta_A^*$
on pullbacks is hence equivalent to
$\rho_A\circ \partial_A=\rho_B\circ \partial_B$.

\medskip

We continue with linear functions. Choose $c\in\Gamma(C)$.
Then for any section $b\in \Gamma(B)$, we get 
\[\nsp{\Theta_B^*(\dr \ell_c)}{\sigma_B^\star(b)}
=\widehat{\nabla_b^*}(\ell_c)=\ell_{\nabla_bc}
\]
and for any $\alpha\in\Gamma(A^*)$:
\[\nsp{\Theta_B^*(\dr \ell_c)}{\alpha^\dagger}
=(\partial_A^*\alpha)^\uparrow(\ell_c)=q_{C^*}^*\langle \alpha, \partial_Ac\rangle.
\]
This shows 
\begin{equation}\label{anchor_on_lift}
\Theta_B^*(\dr
\ell_c)=-\sigma_A^\star(\partial_Ac)+\widetilde{\langle\nabla_\cdot
c,\cdot\rangle}\in\Gamma^l_{C^*}(D\duer A),
\end{equation}
where $\langle\nabla_\cdot c,\cdot\rangle$ is seen as an element
of $\Gamma(\operatorname{Hom}(C^*,B^*))$.
This leads to 
\[\Theta_A\circ\Theta_B^*(\dr \ell_c)(\ell_{c'})=-\ell_{\nabla_{\partial_A(c)}c'}+\ell_{\nabla_{\partial_B(c')}c}
\]
for all $c'\in\Gamma(C)$ and 
\[\Theta_A\circ\Theta_B^*(\dr \ell_c)(q_{C^*}^*f)=-q_{C^*}^*(\rho_A\circ\partial_A(c)f)
\]
for $f\in C^\infty(M)$.
We find similar equations for  $\Theta_B\circ\Theta_A^*(\dr
\ell_c)(\ell_{c'})$ and $\Theta_B\circ\Theta_A^*(\dr
\ell_c)(q^*_{C^*}f)$,
and can conclude that $\Theta_A\circ
\Theta_B^*=-\Theta_B\circ\Theta_A^*$
holds if and only if (M1) and (M2)
are satisfied.
\end{proof}

As a corollary of this proof, we find the following result. Recall that the 
map $\Theta_B \circ \Theta^*_A \co T^*C^*\to TC$ defines a Poisson structure 
on $C^*$ (see (\ref{eq:Poissons}) and the considerations following it).

\begin{corollary}
The Poisson structure on $C^*$ induced by the Lie bialgebroid structure is the 
linear Poisson structure dual to the Lie algebroid structure on $C$ as in 
Remark\, {\rm \ref{rmk:C}}. More explicitly, it is given by
\begin{equation}
\begin{split}
\{\ell_{c_1}, \ell_{c_2}\} 
& = (\Theta_B\circ\Theta^*_A)(q^*_{C^*}\dr\ell_{c_1})(\ell_{c_2})
= \ell_{\nabla_{\partial_A(c_1)}(c_2)  - \nabla_{\partial_B(c_2)}(c_1)}  
= \ell_{[c_1,c_2]},\\
\{\ell_{c_1}, q^*_{C^*}f\} 
& = (\Theta_B\circ\Theta^*_A)(q^*_{C^*}\dr\ell_{c_1})(q^*_{C^*}f)
= q^*_{C^*}(\rho_A(\partial_A(c))(f))\\
\{q^*_{C^*}f_1, q^*_{C^*}f_2\} 
& = (\Theta_B\circ\Theta^*_A)(q^*_{C^*}\dr f_1)(q^*_{C^*}f) = 0. 
\end{split}
\end{equation}
\end{corollary}

\begin{remark}Note that the apparent asymmetry between the structures
over $A$ and $B$ arises from unavoidable choices in the identifications between 
the various duals. The Poisson structure on $C^*$ is nonetheless determined by 
requiring $\partial_A$ and $\partial_B$ to be morphisms of Lie algebroids.
\end{remark}

For the study of (B1) and (B2), we will need the following
lemma. Recall that for a Lie algebroid $A$, the Lie derivative
$\ldr{}\colon \Gamma(A)\times\Gamma(A^*)\to\Gamma(A^*)$ is defined by 
\[ \langle\ldr{a}\alpha, a'\rangle=\rho_A(a)\langle \alpha,
a'\rangle-\langle \alpha, [a,a']\rangle
\]
for all $a,a'\in\Gamma(A)$ and $\alpha\in\Gamma(A^*)$.
\begin{lemma}\label{useful_lie_der}
The Lie derivative $\ldr{}\colon \Gamma_{C^*}(D\duer A)\times
\Gamma_{C^*}(D\duer B)\to \Gamma_{C^*}(D\duer B)$ is given by the following identities:
\begin{equation*}
\begin{split}
\ldr{\beta^\dagger}\alpha^\dagger&=0,\quad 
\ldr{\beta^\dagger}\sigma_B^\star(b)=-\langle b,
\nabla^*_\cdot\beta\rangle^\dagger,\quad 
\ldr{\sigma_A^\star(a)}\alpha^\dagger=\ldr{a}\alpha^\dagger,\\
&\qquad 
\ldr{\sigma_A^\star(a)}\sigma_B^\star(b)=\sigma_B^\star(\nabla_ab)+\widetilde{R(a,\cdot)b}
\end{split}
\end{equation*}
for all $a\in\Gamma(A)$, $b\in\Gamma(B)$, $\alpha\in\Gamma(A^*)$ and $\beta\in\Gamma(B^*)$.
The Lie derivative $\ldr{}\colon \Gamma_{C^*}(D\duer B)\times
\Gamma_{C^*}(D\duer A)\to \Gamma_{C^*}(D\duer A)$ is given by:
\begin{equation*}
\begin{split}
\ldr{\alpha^\dagger}\beta^\dagger&=0,\quad 
\ldr{\alpha^\dagger}\sigma_A^\star(a)=-\langle a,
\nabla^*_\cdot\alpha\rangle^\dagger,\quad 
\ldr{\sigma_A^\star(b)}\beta^\dagger=\ldr{b}\beta^\dagger\\
&\qquad
\ldr{\sigma_A^\star(b)}\sigma_A^\star(a)=\sigma_B^\star(\nabla_ba)+\widetilde{R(b,\cdot)a}.
\end{split}
\end{equation*}
\end{lemma}

Note that in these equations, $R(a,\cdot)b$ is seen as a
  section of $\operatorname{Hom}(C^*,A^*)$ and $R(b,\cdot)a$ is seen
  as a section of $\operatorname{Hom}(C^*,B^*)$.

\begin{proof}
We have 
\begin{equation*}
\begin{split}
\nsp {\beta_2^\dagger}{\ldr{\beta_1^\dagger}\alpha^\dagger}
&=(\partial_B^*\beta_1)^\uparrow\nsp {\beta_2^\dagger}{\alpha^\dagger}
-\nsp {[\beta_1^\dagger, \beta_2^\dagger]}{ \alpha^\dagger}=0\quad\text{ and }\\
\nsp {\sigma_A^\star(a)}{\ldr{\beta_1^\dagger}\alpha^\dagger}&=(\partial_B^*\beta_1)^\uparrow(-q_{C^*}^*\langle\alpha,a\rangle)
+\nsp {\nabla_a^*\beta_1^\dagger}{\alpha^\dagger}=0
\end{split}
\end{equation*}
for arbitrary $\beta_1,\beta_2\in\Gamma(B^*)$, $\alpha\in\Gamma(A^*)$  and $a\in\Gamma(A)$. 
This proves $\ldr{\beta_1^\dagger}\alpha^\dagger=0$. 

\medskip

Then we compute \begin{equation*}
\begin{split}
\nsp {\beta_2^\dagger}{\ldr{\beta_1^\dagger}\sigma_B^\star(b)}
&=(\partial_B^*\beta_1)^\dagger(q_{C^*}^*\langle\beta_2, b\rangle)
-\nsp {[\beta_1^\dagger, \beta_2^\dagger]}{\sigma_B^\star(b)}=0
\end{split}
\end{equation*}
which shows that $\ldr{\beta_1^\dagger}\sigma_B^\star(b)$ is a 
section with values in the core, and 
\begin{equation*}
\begin{split}
\nsp {\sigma_A^\star(a)}{\ldr{\beta_1^\dagger}\sigma_B^\star(b)}&=0
+\nsp {\nabla_a^*\beta_1^\dagger}{\sigma_B^\star(b)}
=q_{C^*}^*\langle b, \nabla_a^*\beta_1\rangle.
\end{split}
\end{equation*}
This proves
$\ldr{\beta_1^\dagger}\sigma_B^\star(b)=-\langle b,
  \nabla_\cdot^*\beta_1\rangle^\dagger$, with $\langle b,
  \nabla_\cdot^*\beta_1\rangle\in\Gamma(A^*)$.
We also find 
\begin{align*}
\nsp {\beta^\dagger}{\ldr{\sigma_A^\star(a_1)}\alpha^\dagger}&= \widehat{\nabla_{a_1}^*}\nsp{\beta^\dagger}
{\alpha^\dagger}-\nsp{\nabla_{a_1}^*\beta^\dagger}{\alpha^\dagger}=0\quad\text{ and }\\
\nsp{\sigma_A^\star(a_2)}{\ldr{\sigma_A^\star(a_1)}\alpha^\dagger}&=-\widehat{\nabla_{a_1}^*}(q_{C^*}^*\langle \alpha,a_2\rangle)
-\nsp {\sigma_A^\star[a_1,a_2]+(R(a_1,a_2)^*)^\dagger}{ \alpha^\dagger}\\
&=-q_{C^*}^*(\rho_A(a_1)\langle\alpha, a_2\rangle-\langle\alpha,[a_1,a_2]\rangle)
=-q_{C^*}^*\langle \ldr{a_1}\alpha,a_2\rangle
\end{align*}
for arbitrary $a_1,a_2\in\Gamma(A)$ and $\alpha\in\Gamma(A^*)$. This proves
the equality
$\ldr{\sigma_A^\star(a_1)}\alpha^\dagger=\ldr{a_1}\alpha^\dagger$.
\medskip

The identity
\begin{equation*}
\begin{split}
\nsp {\beta^\dagger}{\ldr{\sigma_A^\star(a_1)}\sigma_B^\star(b)}
&=\widehat{\nabla_{a_1}^*}q_{C^*}^*\langle
\beta, b\rangle
-\nsp{\nabla_{a_1}^*\beta^\dagger}{\sigma_B^\star(b)} \\
&=q_{C^*}^*\left(\rho_A(a_1) \langle
\beta,b\rangle-\langle\nabla^*_{a_1}\beta, b\rangle
\right)=q_{C^*}^*\langle \beta, \nabla_{a_1}b\rangle
\end{split}
\end{equation*}
shows that $\ldr{\sigma_A^\star(a_1)}\sigma_B^\star(b)$ is the
sum of $\sigma_B^\star(\nabla_{a_1}b)$ with a section with values in
the core. To find out this core term, we finally compute
\begin{equation*}
\begin{split}
\nsp {\sigma_A^\star(a_2)}{\ldr{\sigma_A^\star(a_1)}\sigma_B^\star(b)}&=0-\nsp
{\sigma_A^*[a_1,a_2]+\widetilde{R(a_1,a_2)^*}}{\sigma_B^\star(b)}\\
&=-\ell_{R(a_1,a_2)(b)}.
\end{split}
\end{equation*}
This shows that
$\ldr{\sigma_A^\star(a_1)}\sigma_B^\star(b)=\sigma_B^\star(\nabla_{a_1}b)+\widetilde{R(a_1,\cdot)b}$.

\medskip

The formulas describing the Lie derivative $\ldr{}\colon \Gamma_{C^*}(D\duer B)\times
\Gamma_{C^*}(D\duer A)\to \Gamma_{C^*}(D\duer A)$ can be verified in the
same manner.
\end{proof}

\begin{proposition}
Condition {\rm (B2)} on $\mathcal S$ and $\mathcal R$ is equivalent to 
{\rm (M3)}, {\rm (M4)}, {\rm (M5)} and {\rm (M6)}.
\end{proposition}

\begin{proof}
The idea of this proof is to check (B2) on linear and core sections in $\mathcal S$ and
$\mathcal R$, and on linear and $q_{C^*}$-pullback functions on $C^*$.
We start with core sections. Choose $\alpha\in\Gamma(A^*)$ and
  $\beta\in\Gamma(B^*)$. 
We have $[\Theta_B(\alpha^\dagger), \Theta_A(\beta^\dagger)]=[(\partial_A^*\alpha)^\uparrow,(\partial_B^*\beta)^\uparrow]=0$.
By Lemma \ref{useful_lie_der} and with
$\nsp{\beta^\dagger}{\alpha^\dagger}=0$, we find that (B2) is
trivially satisfied 
on $\alpha^\dagger, \beta^\dagger$ and any element of $C^\infty(C^*)$.

Now choose $a\in\Gamma(A)$, $\alpha\in\Gamma(A^*)$.
Using Lemma \ref{useful_lie_der} we find for all $F\in C^\infty(C^*)$
\begin{equation*}
\begin{split}
&[\Theta_B(\alpha^\dagger),\Theta_A(\sigma_A^\star(a))](F)-\Theta_A(\ldr{\alpha^\dagger}\sigma_A^\star(a))(F)+\Theta_B(\ldr{\sigma_A^\star(a)}\alpha^\dagger)(F)\\
&\qquad -\Theta_A(\dr_{D\duer
  B}F)\nsp{\sigma_A^\star(a)}{\alpha^\dagger}\\
=&-(\nabla_a^*(\partial_A^*\alpha))^\uparrow(F)+\langle a,\nabla_{\partial_B\cdot}^*\alpha\rangle^\uparrow(F)+(\partial_A^*\ldr{a}\alpha)^\uparrow(F)+(\Theta_A\circ\Theta_B^*\dr
  F)q_{C^*}^*\langle\alpha, a\rangle.
\end{split}
\end{equation*}
In particular, for $F=q_{C^*}^*f$, $f\in
  C^\infty(M)$, this is 
$0+(\partial_B^*\rho_B^*\dr
  f)^\uparrow q_{C^*}^*\langle\alpha, a\rangle=0$
by \eqref{anchor_on_core}
and for $F=\ell_c$, $c\in\Gamma(C)$, this is 
\begin{align*}
q_{C^*}^*(-\langle \nabla_a^*(\partial_A^*\alpha)),c\rangle+\langle a,
  \nabla_{\partial_Bc}^*\alpha\rangle+\langle\partial_A^*\ldr{a}\alpha,
  c\rangle-(\rho_B\circ\partial_B(c))\langle\alpha,a\rangle)
\end{align*}
by \eqref{anchor_on_lift}. But this equals
$q_{C^*}^*\left(\left\langle \alpha, \partial_A\left(\nabla_ac\right)
-\nabla_{\partial_Bc} a
- [a, \partial_Ac]\right\rangle\right)$.
This shows that (B2) is in this case equivalent to (M3). In the same
manner,
(B2) on $\beta^\dagger\in \mathcal S$ for $\beta\in\Gamma(B^*)$, 
$\sigma_B^\star(b)\in\mathcal R$ for $b\in\Gamma(B)$ and $F\in
C^\infty(C^*)$ is equivalent to (M4).  

Now choose $a\in\Gamma(A)$ and $b\in\Gamma(B)$. (B2) on
$\sigma_A^\star(a)$, $\sigma_B^\star(b)$ and $q_{C^*}^*f$, $f\in C^\infty(M)$,
is
\begin{equation*}
q_{C^*}^*\left(\left(
[\rho_B(b),\rho_A(a)]-\rho_A(\nabla_ba)+\rho_B(\nabla_ab)\right)(f) \right)=0
\end{equation*}
by Lemma \ref{useful_lie_der}.  This is (M5).  Finally we compute (B2)
on $\sigma_A^\star(a)$, $\sigma_B^\star(b)$ and $\ell_c$, for
$c\in\Gamma(C)$.  This is
\begin{align*}
\ell_{\nabla_b\nabla_ac-\nabla_a\nabla_bc}-\ell_{\nabla_{\nabla_ba}c+R_{BA}(b,\partial_Bc)a}
+\ell_{\nabla_{\nabla_ab}c+R_{AB}(a,\partial_Ac)b}=0
\end{align*}
Lemma \ref{useful_lie_der}.
We find hence that (B2) on $\sigma_A^\star(a)$, $\sigma_B^\star(b)$ and
$\ell_c$ is equivalent to (M6).
\end{proof}

We conclude the proof of Theorem \ref{main} with the study of (B1) on 
linear and core sections.
\begin{proposition}
Assume that {\rm (M5)} is given. Condition {\rm (B1)} on elements of 
$\mathcal S$ and $\mathcal R$ is equivalent to {\rm (M7)}, {\rm (M8)} and {\rm (M9)}.
\end{proposition}

In the proof of this proposition, we will use the following formulas.
Let $A$ and $A^*$ be a pair of Lie algebroids in duality. Then, for all
$a\in\Gamma(A)$ and $\alpha_1,\alpha_2\in\Gamma(A^*)$:
\[(\dr_{A^*}a)(\alpha_1,\alpha_2)=\rho_{A^*}(\alpha_1)\langle\alpha_2,
a\rangle-\rho_{A^*}(\alpha_2)\langle\alpha_1, a\rangle-\langle
[\alpha_1,\alpha_2]_{A^*}, a\rangle.
\]
For all $a_1,a_2\in\Gamma(A)$ and $\alpha_1,\alpha_2\in\Gamma(A^*)$,
we have 
\begin{equation*}
\begin{split}
 [\dr_{A^*}a_1,a_2]_A(\alpha_1,\alpha_2)&=-(\ldr{a_2}\dr_{A^*}a_1)(\alpha_1,\alpha_2)\\
&=-\ldr{\rho_A(a_2)}(\dr_{A^*}a_1(\alpha_1,\alpha_2))+\dr_{A^*}a_1(\ldr{a_2}\alpha_1,
\alpha_2)+\dr_{A^*}a_1(\alpha_1, \ldr{a_2}\alpha_2).
\end{split}
\end{equation*}

\begin{proof}
First choose $\alpha_1,\alpha_2\in\Gamma(A^*)$. We have $\dr_{D\duer
    A}[\alpha_1^\dagger,\alpha_2^\dagger]=0$. For
  $\beta_1,\beta_2\in\Gamma(B^*)$ and $a_1,a_2\in\Gamma(A)$, we find
  using Lemma \ref{useful_lie_der}
\begin{equation*}
\begin{split}
[\dr_{D\duer A}\alpha_1^\dagger, \alpha_2^\dagger](\beta_1^\dagger,\beta_2^\dagger)
=&\,-(\partial_A^*\alpha_2)^\uparrow(\dr_{D\duer A}\alpha_1^\dagger(\beta_1^\dagger,\beta_2^\dagger))
+\dr_{D\duer A}\alpha_1^\dagger(\ldr{\alpha_2^\dagger}\beta_1^\dagger,\beta_2^\dagger)\\
&\qquad +\dr_{D\duer A}\alpha_1^\dagger(\beta_1^\dagger,\ldr{\alpha_2^\dagger}\beta_2^\dagger)=0,\\
[\dr_{D\duer A}\alpha_1^\dagger, \alpha_2^\dagger](\sigma_A^\star(a_1),\beta_2^\dagger)
=&\,-(\partial_A^*\alpha_2)^\dagger (\dr_{D\duer A}\alpha_1^\dagger(\sigma_A^\star(a_1),\beta_2^\dagger))
+\dr_{D\duer A}\alpha_1^\dagger(\ldr{\alpha_2^\dagger}\sigma_A^\star(a_1),\beta_2^\dagger)\\
&\qquad +\dr_{D\duer A}\alpha_1^\dagger(\sigma_A^\star(a_1),\ldr{\alpha_2^\dagger}\beta_2^\dagger)
=0 
\end{split}
\end{equation*}
and
\begin{equation*}
\begin{split}
&[\dr_{D\duer A}\alpha_1^\dagger, \alpha_2^\dagger](\sigma_A^\star(a_1),\sigma_A^\star(a_2))\\
=&\,-(\partial_A^*\alpha_2)^\uparrow (\dr_{D\duer A}\alpha_1^\dagger(\sigma_A^\star(a_1),\sigma_A^\star(a_2)))
+\dr_{D\duer A}\alpha_1^\dagger(\ldr{\alpha_2^\dagger}\sigma_A^\star(a_1),\sigma_A^\star(a_2))\\
&\qquad+\dr_{D\duer
  A}\alpha_1^\dagger(\sigma_A^\star(a_1),\ldr{\alpha_2^\dagger}\sigma_A^\star(a_2))\\
=&\,-(\partial_A^*\alpha_2)^\uparrow q_{C^*}^*\left(-\rho_A(a_1)\langle\alpha_1,a_2\rangle+
  \rho_A(a_2)\langle\alpha_1,a_1\rangle
+\langle\alpha_1,[a_1,a_2]\rangle
\right)\\
&\qquad +\dr_{D\duer A}\alpha_1^\dagger(-\langle a_1,\nabla_\cdot^*\alpha_2\rangle^\dagger,\sigma_A^\star(a_2))+\dr_{D\duer
  A}\alpha_1^\dagger(\sigma_A^\star(a_1),-\langle
  a_2,\nabla_\cdot^*\alpha_2\rangle^\dagger)=0.
\end{split}
\end{equation*}
Thus, we have $\dr_{D\duer
  A}[\alpha_1^\dagger,\alpha_2^\dagger]=0=\left[\dr_{D\duer
    A}\alpha_1^\dagger,
  \alpha_2^\dagger\right]+\left[\alpha_1^\dagger, \dr_{D\duer
    A}\alpha_2^\dagger\right]$.  Choose now $\alpha\in\Gamma(A^*)$ and
$b\in\Gamma(B)$.  We have $\dr_{D\duer
  A}[\sigma_B^\star(b),\alpha^\dagger]=\dr_{D\duer
  A}(\nabla_b^*\alpha)^\dagger$, and so in particular
$\dr_{D\duer
  A}[\sigma_B^\star(b),\alpha^\dagger](\beta_1^\dagger,\beta_2^\dagger)=0$,
$\dr_{D\duer
  A}[\sigma_B^\star(b),\alpha^\dagger](\sigma_A^\star(a_1),\beta_2^\dagger)=0$
and 
\begin{equation*}
\begin{split}
&\dr_{D\duer
  A}[\sigma_B^\star(b),\alpha^\dagger](\sigma_A^\star(a_1),\sigma_A^\star(a_2))\\
&=q_{C^*}^*\left(
-\rho_A(a_1)\langle \nabla_b^*\alpha,a_2\rangle
+\rho_A(a_2)\langle \nabla_b^*\alpha,a_1\rangle
+\langle\nabla_b^*\alpha,[a_1,a_2]\rangle
\right)\\
&=-q_{C^*}^*\left(\dr_{A}(\nabla_b^*\alpha)(a_1,a_2)
\right).
\end{split}
\end{equation*}
On the other hand, we can check as above that
$[\dr_{D\duer A}
\sigma_B^\star(b),\alpha^\dagger](\beta_1^\dagger,\beta_2^\dagger)=0$,\linebreak
$[\dr_{D\duer A}
\sigma_B^\star(b),\alpha^\dagger](\sigma_A^\star(a_1),\beta_2^\dagger)=0$
and 
\begin{equation*}
\begin{split}
&[\dr_{D\duer A}
\sigma_B^\star(b),\alpha^\dagger](\sigma_A^\star(a_1),\sigma_A^\star(a_2))\\
=&\,
-(\partial_A^*\alpha)^\uparrow\left(\dr_{D\duer A} \sigma_B^\star(b)
  (\sigma_A^\star(a_1),\sigma_A^\star(a_2))\right) \\
&\quad
+\dr_{D\duer A} \sigma_B^\star(b)(\ldr{\alpha^\dagger}\sigma_A^\star(a_1),\sigma_A^\star(a_2))+\dr_{D\duer A}
\sigma_B^\star(b)(\sigma_A^\star(a_1),\ldr{\alpha^\dagger}\sigma_A^\star(a_2))\\
=&\,(\partial_A^*\alpha)^\dagger\nsp
  {[\sigma_A^\star(a_1),\sigma_A^\star(a_2)]}{ \sigma_B^\star(b)}\\
&\quad
+\dr_{D\duer A} \sigma_B^\star(b)(-\langle
  a_1,\nabla_\cdot^*\alpha\rangle^\dagger,\sigma_A^\star(a_2))+\dr_{D\duer A}
\sigma_B^\star(b)(\sigma_A^\star(a_1),-\langle
  a_2,\nabla_\cdot^*\alpha\rangle^\dagger)\\
=&
q_{C^*}^*\bigl(\langle\partial_A^*\alpha, R(a_1,a_2)b\rangle
+\rho_A(a_2) \langle a_1,\nabla_b^*\alpha\rangle
-\langle b, \nabla_{a_2}^*\langle
  a_1, \nabla_\cdot^*\alpha\rangle\rangle\\
&\quad -\rho_A(a_1) \langle a_2,\nabla_b^*\alpha\rangle
+\langle b, \nabla_{a_1}^*\langle
  a_2, \nabla_\cdot^*\alpha\rangle\rangle\bigr)\\
=&q_{C^*}^*\bigl(\langle\partial_A^*\alpha, R(a_1,a_2)b\rangle
+\langle
  a_1,\nabla_{\nabla_{a_2}b}^*\alpha\rangle 
-\langle a_2,\nabla_{\nabla_{a_1}b}^*\alpha\rangle\bigr).
\end{split}
\end{equation*}
Recall \eqref{bracket_core_linear} and \eqref{bracket_core_linear_2}.
Using this, we finally get
$\left[ \sigma_B^\star(b),\dr_{D\duer
    A}\alpha^\dagger\right](\beta_1^\dagger,\beta_2^\dagger)=0$, 
\begin{equation*}
\begin{split}
&\left[ \sigma_B^\star(b),\dr_{D\duer A}\alpha^\dagger\right](\sigma_A^\star(a_1),\beta_2^\dagger)=
\widehat{\nabla_b^*}\left(\dr_{D\duer A} \alpha^\dagger
  (\sigma_A^\star(a_1),\beta_2^\dagger)\right)\\
&\quad \qquad \qquad -\dr_{D\duer A} \alpha^\dagger
(\sigma_A^\star\left(\nabla_ba_1\right)+\widetilde{R(b,\cdot)a_1},\beta_2^\dagger)
-\dr_{D\duer A}\alpha^\dagger
(\sigma_A^\star(a_1),\ldr{b}\beta_2^\dagger)=0
\end{split}
\end{equation*}
and
\begin{align*}
&\left[ \sigma_B^\star(b),\dr_{D\duer A}\alpha^\dagger\right](\sigma_A^\star(a_1),\sigma_A^\star(a_2))\\
=&\widehat{\nabla_b^*}\left(\dr_{D\duer A} \alpha^\dagger
  (\sigma_A^\star(a_1),\sigma_A^\star(a_2))\right)-\dr_{D\duer A} \alpha^\dagger (\sigma_A^\star\left(\nabla_ba_1\right)+\widetilde{R(b,\cdot)a_1},\sigma_A^\star(a_2))\\
&\quad -\dr_{D\duer A}\alpha^\dagger
(\sigma_A^\star(a_1),\sigma_A^\star\left(\nabla_ba_2\right)+\widetilde{R(b,\cdot)a_2})\\
=&\,q_{C^*}^*\bigl(-\rho_B(b)(\dr_A\alpha(a_1,a_2))
+\rho_A(\nabla_ba_1)\langle\alpha,
a_2\rangle-\rho_A(a_2)\langle\nabla_ba_1,\alpha\rangle\\
&\quad -\langle\alpha, [\nabla_ba_1, a_2]\rangle+\rho_A(a_1)\langle
\nabla_ba_2,\alpha\rangle-\rho_A(\nabla_ba_2)\langle a_1,\alpha\rangle
-\langle \alpha, [a_1, \nabla_ba_2]\rangle
\bigr).
\end{align*}
We hence find that 
\[     \dr_{D\duer A}[\sigma_B^\star(b),\alpha^\dagger]  = \left[\dr_{D\duer A} \sigma_B^\star(b),\alpha^\dagger\right]+\left[ \sigma_B^\star(b),\dr_{D\duer A}\alpha^\dagger\right]
\]
if and only if 
\begin{equation*}
\begin{split}
&\dr_{A}(\nabla_b^*\alpha)(a_1,a_2)
+\langle\partial_A^*\alpha, R(a_1,a_2)b\rangle
+\langle
  a_1,{\nabla_{\nabla_{a_2}b}}^*\alpha\rangle 
-\langle a_2,\nabla_{\nabla_{a_1}b}^*\alpha\rangle\\
&-\rho_B(b)(\dr_A\alpha(a_1,a_2))
+\rho_A(\nabla_ba_1)\langle\alpha,
a_2\rangle-\rho_A(a_2)\langle\nabla_ba_1,\alpha\rangle -\langle\alpha, [\nabla_ba_1, a_2]\rangle\\
&\quad +\rho_A(a_1)\langle
\nabla_ba_2,\alpha\rangle-\rho_A(\nabla_ba_2)\langle a_1,\alpha\rangle
-\langle \alpha, [a_1, \nabla_ba_2]\rangle=0
\end{split}
\end{equation*}
for all $a_1,a_2\in\Gamma(A^*)$. This is 
\begin{equation*}
\begin{split}
&\langle\alpha, \partial_AR(a_1,a_2)b+\nabla_b[a_1,a_2]-[\nabla_ba_1,a_2]-[a_1,\nabla_ba_2]+\nabla_{\nabla_{a_1}b}a_2-\nabla_{\nabla_{a_2}b}a_1\rangle\\
+&\left([\rho_A(a_1),\rho_B(b)]-\rho_B(\nabla_{a_1}b)+\rho_A(\nabla_ba_1)\right)\langle \alpha, a_2\rangle\\
-&\left([\rho_A(a_2),\rho_B(b)]-\rho_B(\nabla_{a_2}b)+\rho_A(\nabla_ba_2)\right)\langle \alpha, a_1\rangle=0.
\end{split}
\end{equation*}
Hence, using (M5) twice, we obtain (M7) since $\alpha$ was arbitrary.

\medskip
  We conclude the proof of the theorem with the most technical formula. Choose
  $b_1,b_2\in\Gamma(B)$. We want to study the equation
\begin{equation}\label{last_condition}
\dr_{D\duer A}\left[\sigma_B^\star(b_1),\sigma_B^\star(b_2)\right]=\left[\dr_{D\duer A}
\sigma_B^\star(b_1),\sigma_B^\star(b_2)\right]+\left[ \sigma_B^\star(b_1), \dr_{D\duer A} \sigma_B^\star(b_2)\right].
\end{equation}
We have $\dr_{D\duer
  A}\left[\sigma_B^\star(b_1),\sigma_B^\star(b_2)\right]=\dr_{D\duer
  A}\left(\sigma_B^\star[b_1,b_2]+{R(b_1,b_2)^*}^\dagger\right)$ and we
find easily that both sides of \eqref{last_condition} vanish
on $\beta_1^\dagger,\beta_2^\dagger$, for $\beta_1, \beta_2\in\Gamma(B^*)$.
We have for $a\in\Gamma(A)$ and $\beta\in\Gamma(B^*)$:
\begin{equation*}
\begin{split}
&\dr_{D\duer
  A}\left(\sigma_B^\star[b_1,b_2]+\widetilde{R_{BA}(b_1,b_2)^*}\right)(\sigma_A^\star(a),\beta^\dagger)\\
=&\,
q_{C^*}^*\left(\rho_A(a)\langle [b_1,b_2],\beta\rangle
+\langle \partial_B^*\beta, R(b_1,b_2)a\rangle
-\langle[b_1,b_2], \nabla_a^*\beta\rangle
\right)\\
=&\,q_{C^*}^*\left(\langle \nabla_a[b_1,b_2],\beta\rangle
+\langle \partial_B^*\beta, R(b_1,b_2)a\rangle
\right)
\end{split}
\end{equation*}
and 
\begin{align*}
&\left[\dr_{D\duer A} \sigma_B^\star(b_1),\sigma_B^\star(b_2)\right](\sigma_A^\star(a),\beta^\dagger)=
-\widehat{\nabla_{b_2}^*}\left(\dr_{D\duer
    A}\sigma_B^\star(b_1)(\sigma_A^\star(a),\beta^\dagger)\right)\\
&\qquad\qquad+\dr_{D\duer
    A}\sigma_B^\star(b_1)\left(\sigma_A^\star\left(\nabla_{b_2}a\right)+\widetilde{R(b_2,\cdot)a},\beta^\dagger\right)\\
&\qquad\qquad+\dr_{D\duer
    A}\sigma_B^\star(b_1)\left(\sigma_A^\star(a),\ldr{b_2}\beta^\dagger\right)\\
\overset{\eqref{bracket_core_linear_2}}{=}&q_{C^*}^*\bigl(-\rho_B(b_2)\rho_A(a)\langle
b_1,\beta\rangle+\rho_B(b_2)\langle
b_1,\nabla_a^*\beta\rangle+\rho_A(\nabla_{b_2}a)\langle b_1, \beta\rangle
-\cancel{\langle\partial_B^*\beta, R(b_2,b_1)a\rangle}\\
&\,-\langle b_1, \nabla_{\nabla_{b_2}a}^*\beta\rangle
+\cancel{\langle \partial_B^*\beta, R(b_2,b_1)a\rangle}+\rho_A(a)\langle b_1,\ldr{b_2}\beta\rangle-\langle b_1,\nabla_a^*\ldr{b_2}\beta\rangle
\bigr)\\
=\,&q_{C^*}^*\bigl(-\rho_B(b_2)\langle
\nabla_ab_1,\beta\rangle+\langle \nabla_{\nabla_{b_2}a} b_1, \beta\rangle+\langle \nabla_a b_1,\ldr{b_2}\beta\rangle
\bigr)\\
=\,&q_{C^*}^*\bigl(-\langle
[b_2,\nabla_ab_1],\beta\rangle+\langle \nabla_{\nabla_{b_2}a}b_1, \beta\rangle\bigr).
\end{align*}
Thus, we find that the two sides of \eqref{last_condition} are equal on
$(\sigma_A^\star(a),\beta^\dagger)$ if and only if (M8) is satisfied.

Finally we consider $a_1,a_2\in\Gamma(A)$. 
We have 
\begin{align*}
&\dr_{D\duer
  A}\left(\sigma_B^\star[b_1,b_2]+\widetilde{R(b_1,b_2)^*}\right)(\sigma_A^\star(a_1),\sigma_A^\star(a_2))\\
=&\,-\widehat{\nabla_{a_1}^*}\ell_{R(b_1,b_2)a_2}
+\widehat{\nabla_{a_2}^*}\ell_{R(b_1,b_2)a_1}
+\ell_{R(b_1,b_2)[a_1,a_2]-R(a_1,a_2)[b_1,b_2]}=\ell_c
\end{align*}
where $c= {-\nabla_{a_1}(R(b_1,b_2)a_2)+\nabla_{a_2}(R(b_1,b_2)a_1)+R(b_1,b_2)[a_1,a_2]
-R(a_1,a_2)[b_1,b_2]}\in \Gamma(C)$, 
and 
\begin{align*}
&\left[\dr_{D\duer A}
  \sigma_B^\star(b_1),\sigma_B^\star(b_2)\right](\sigma_A^\star(a_1),\sigma_A^\star(a_2))=\,
-\widehat{\nabla_{b_2}^*}\left(\dr_{D\duer
    A}\sigma_B^\star(b_1)(\sigma_A^\star(a_1),\sigma_A^\star(a_2))\right)\\
&\qquad\qquad+\dr_{D\duer
    A}\sigma_B^\star(b_1)\left(\sigma_A^\star\left(\nabla_{b_2}a_1\right)+\widetilde{R(b_2,\cdot)a_1},\sigma_A^\star(a_2)\right)\\
&\qquad\qquad+\dr_{D\duer
    A}\sigma_B^\star(b_1)\left(\sigma_A^\star(a_1),
    \sigma_A^\star\left(\nabla_{b_2}a_2\right)+\widetilde{R(b_2,\cdot)a_2}\right)\\
\overset{\eqref{bracket_core_linear}}{=}\,&\widehat{\nabla_{b_2}^*}\ell_{R_{AB}(a_1,a_2)b_1}-\widehat{\nabla_{a_2}^*}\ell_{R(b_2,b_1)a_1}
-\nsp{\widetilde{R(\nabla_{b_2}a_1,
  a_2)^*}-\widetilde{{\nabla_{a_2}^{\rm
      Hom}}(R(b_2,\cdot)a_1)}}{\sigma_B^\star(b_1)}\\
&+\widehat{\nabla_{a_1}^*}\ell_{R(b_2,b_1)a_2}
+\nsp{ \widetilde{R(\nabla_{b_2}a_2, a_1)^*}
 -\widetilde{{\nabla_{a_1}^{\rm Hom}}(R(b_2,\cdot)a_2)}}{\sigma_B^\star(b_1)}\\
=\,\,&%\ell_{c_1}+\ell_{c_2} = 
\ell_{c_1+c_2},
\end{align*}
where
\begin{align*}
c_1&= {\nabla_{b_2}(R(a_1,a_2)b_1)-\nabla_{a_2}(R(b_2,b_1)a_1)-R(\nabla_{b_2}a_1,
  a_2)b_1+\langle b_1, {\nabla_{a_2}^{\rm Hom}}(R(b_2,\cdot)a_1)\rangle}\\
&=\nabla_{b_2}(R(a_1,a_2)b_1)+R(\nabla_{a_2} b_1,b_2)a_1-R(\nabla_{b_2}a_1,
  a_2)b_1,\\
c_2&= {\nabla_{a_1}(R(b_2,b_1)a_2)+R(\nabla_{b_2}a_2,
  a_1)b_1-\langle b_1, {\nabla_{a_1}^{\rm Hom}}R(b_2,\cdot)a_2 \rangle}\\
&=R(\nabla_{a_1} b_1, b_2)a_2-R(a_1,\nabla_{b_2}a_2)b_1.
\end{align*}

Hence, we find that the two sides of \eqref{last_condition}
coincides on $(\sigma_A^\star(a_1),\sigma_A^\star(a_2))$ if and only
if (M9) is satisfied.
\end{proof}

\newpage

\def\cprime{$'$} \def\polhk#1{\setbox0=\hbox{#1}{\ooalign{\hidewidth
  \lower1.5ex\hbox{`}\hidewidth\crcr\unhbox0}}} \def\cprime{$'$}
\providecommand{\bysame}{\leavevmode\hbox to3em{\hrulefill}\thinspace}
\providecommand{\MR}{\relax\ifhmode\unskip\space\fi MR }
\providecommand{\MRhref}[2]{%
  \href{http://www.ams.org/mathscinet-getitem?mr=#1}{#2}
}
\providecommand{\href}[2]{#2}

\end{document}